\documentclass[12pt,oneside,reqno,a4paper,dvipdfmx]{amsart}

\usepackage[margin=1in]{geometry}

\usepackage{amsmath,amssymb}
\usepackage{amscd}
\usepackage{color}
\theoremstyle{plain}
\newtheorem{thm}{Theorem*}[section]
\newtheorem{theorem}[thm]{Theorem}

\newtheorem{lemma}[thm]{Lemma}
\newtheorem{corollary}[thm]{Corollary}
\newtheorem{proposition}[thm]{Proposition}
\theoremstyle{definition}

\newtheorem{remark}[thm]{Remark}

\newtheorem{definition}[thm]{Definition}

\newtheorem{example}[thm]{Example}

\numberwithin{equation}{section}
\newcommand{\0}{{\mathcal O}}
\newcommand{\sA}{{\mathcal A}}
\newcommand{\sB}{{\mathcal B}}

\newcommand{\sM}{{\mathcal M}}

\newcommand{\sV}{{\mathcal V}}
\newcommand{\sW}{{\mathcal W}}

\newcommand{\C}{{\mathbb C}}

\newcommand{\BP}{{\mathbb P}}

\newcommand{\Q}{{\mathbb Q}}
\newcommand{\R}{{\mathbb R}}

\newcommand{\Z}{{\mathbb Z}}

\newcommand{\Aut}{{\rm Aut}}

\newcommand{\rank}{{\rm rank}}

\newcommand{\id}{{\rm id}}
\newcommand{\tr}{{\rm trace}}
\newcommand{\disc}{{\rm disc}}

\newcommand{\OP}[1]{\operatorname{#1}}
\newcommand{\MF}[1]{\mathfrak{#1}}

\newcommand{\DD}{\Delta}
\newcommand{\DDz}{\Delta_0}
\newcommand{\ANG}[1]{\langle #1 \rangle}
\newcommand{\ANGz}[1]{\langle #1 \rangle^{}_{0}}
\newcommand{\SqMat}[4]{\begin{pmatrix}#1&#2\\#3&#4\end{pmatrix}}
\newcommand{\SqMatSmall}[4]{\begin{smallmatrix}#1&#2\\#3&#4\end{smallmatrix}}
\newcommand{\ROOT}{{\rm Root}}
\newcommand{\VE}[2]{#1^{[#2]}}
\newcommand{\GAM}{{D}} 
\newcommand{\GAMu}{{D}}
\newcommand{\PGB}[1]{\Gamma^\times(#1)}
\newcommand{\PSB}[1]{\Gamma^1(#1)}
\newcommand{\SQ}[1]{#1^\times/(#1^\times)^2}
\newcommand{\CONGa}{\,\,{\displaystyle\mathop{\rightarrow}^\sim}\,\,} 

\newcommand{\clfC}{\mathcal{B}}
\newcommand{\trzero}{\widetilde{B}}

\newcommand{\trzeroQ}{\widetilde{B}_\Q}
\newcommand{\OSUM}{\bot}
\newcommand{\LL}[1]{(#1)} 
\newcommand{\LAT}{{\mathcal{L}}}
\newcommand{\LATa}{{\mathcal{N}}}
\newcommand{\LATb}{{\mathcal{K}}}
\newcommand{\LatK}{\Lambda_{\scriptstyle \rm{K3}}} 
\newcommand{\CLF}[1]{\OP{Cl}(#1)} 
\newcommand{\CLFP}[1]{\OP{Cl}^+(#1)} 
\newcommand{\CLFM}[1]{\OP{Cl}^-(#1)} 
\newcommand{\CLFPM}[1]{\OP{Cl}^{\pm}(#1)} 
\newcommand{\LatM}{L^\flat} 
\newcommand{\OverQ}{_{\Q}}
\newcommand{\CLFinvo}[1]{#1^*}
\newcommand{\dE}{E^\vee} 
\newcommand{\NORM}{\mathcal{N}}
\newcommand{\CONE}{\Omega^+}

\def\Hom{\mathop{\rm Hom}\nolimits}

\title[Free automorphism groups of K3 surfaces]%
{Free automorphism groups of K3 surfaces with Picard number $3$}
\author{Kenji Hashimoto and Kwangwoo Lee}
\date{\today}

\begin{document}

\setcounter{tocdepth}{1}

\begin{abstract}
It is known that
 the automorphism group of any projective K3 surface is
 finitely generated \cite{St}.
In this paper,
 we consider a certain kind of K3 surfaces with Picard number $3$
 whose automorphism groups are isomorphic to
 congruence subgroups of the modular group $PSL_2(\Z)$.
In particular,
 we show that a free group of arbitrarily large rank
 appears as the automorphism group of such a K3 surface.
\end{abstract}

\maketitle

\tableofcontents

\section{Introduction} \label{sect:intro}

In \cite{St}, Sterk showed that the automorphism group of
 a projective K3 surface over the complex numbers
 is finitely generated.
In this paper, we consider a class of projective K3 surfaces
 whose automorphism groups are free groups
 such that their ranks
 are not bounded above.
In particular, we study automorphism groups of K3 surfaces $X_n$
 whose Picard lattices are isomorphic to $U(n) \OSUM \LL{-2n}$,
 $n\in \Z_{\geq 2}$.
(See Proposition \ref{prop:auto-k3-Xn}.)

A typical example of such K3 surfaces is Wehler's example \cite{We},
 where $n=2$.
Let $X_2$ be a very general $(2,2,2)$-hypersurface in
 $\BP^1\times \BP^1\times \BP^1$.
Then $X_2$ has Picard number $3$ and its Picard lattice
 is generated by the pull-backs $H_i$ of a point by the $i$th projection
 $p_i \colon X_2 \rightarrow \BP^1$ ($i=1,2,3$).
The intersection matrix with respect to the basis $(H_i)$ is
\begin{equation*}
 \begin{pmatrix} 0& 2&2\\2&0&2\\ 2&2&0 \end{pmatrix}
 \cong U(2) \OSUM \LL{-4}.
\end{equation*}
The double covering
 $p_1 \times p_2 \colon X_2 \rightarrow \BP^1 \times \BP^1$
 induces an involution $\iota_{12}$ on $X_2$
 as the covering transformation.
Similarly we obtain involutions $\iota_{13}$ and $\iota_{23}$.
In fact these three involutions generate the automorphism group of $X_2$:
\begin{equation*}
 \Aut(X_2) \cong 
 \{\alpha \in PGL_2(\Z) \bigm| \alpha \equiv \pm I \bmod{2} \}
 \cong C_2 * C_2 * C_2.
\end{equation*}
Here each $C_2$ is a cyclic group of order $2$
 generated by each $\iota_{ij}$.
See \cite{We} for the details of this result.

{\bf Notations.}
Let $L$ be an even lattice.
Then, $q(L)$ and $O(q(L))$ denote the discriminant form of $L$
 and its isomorphism group, respectively;
 $\CLFP{L}$ and $\CLFM{L}$ denote the even and odd parts
 of the Clifford algebra $\CLF{L}$ of $L$, respectively.
(We summarize our convention in Section \ref{subsect:convention}.)

In what follows, we overview
 our main theorem (Theorem \ref{thm:MAIN}) and its applications.
These can be considered as
 a generalization (or an analogy) of Wehler's result.

\begin{theorem}[see Theorem {\ref{thm:MAIN}} for details]
Let $L$ be a non-degenerate even lattice of rank $3$.
Then there is a natural isomorphism
\begin{equation}
 O_\GAM(L) \cong
 \CLFPM{L}^\times / \{ \pm 1 \},
\end{equation}
 where $O_\GAM(L) := \OP{Ker}(O(L) \rightarrow O(q(L)))$
 is the discriminant kernel for $O(L)$ and
 $\CLFPM{L}^\times$ is given as
\begin{equation}
 \CLFPM{L}^\times =
 \{ \alpha \in \CLFP{L} \cup \CLFM{L} \bigm|
 N \alpha = \pm 1 \}.
\end{equation}
Moreover, under this isomorphism, we have
\begin{align}
 & SO_\GAM(L) {} \cong
 \OP{Cl}^+(L)^\times / \{ \pm 1 \}, \quad
 \OP{Cl}^+(L)^\times =
 \{ \alpha \in \OP{Cl}^+(L) \bigm| N \alpha = \pm 1 \}; \\
 & SO^+_\GAM(L)
 {} \cong \OP{Cl}^+(L)^1 / \{ \pm 1 \}, \quad
 \OP{Cl}^+(L)^1 =
 \{ \alpha \in \OP{Cl}^+(L) \bigm| N \alpha = 1 \}.
\end{align}
\end{theorem}

\begin{corollary}
Let $X$ be a K3 surface with Picard number $3$.
Suppose that the Picard lattice $S_X$ of $X$
 is non-degenerate.
Moreover, we assume that
 $X$ contains no $(-2)$-curves, that is,
 $S_X$ contains
 no elements with self-intersection number $=-2$.
Then
\begin{equation}
 \Aut(X) \cong
  \{ g \in O^+(S_X) \bigm| \bar{g} = \pm \id \}
 \cong O_\GAM(S_X)
 \cong \CLFPM{S_X}^\times / \{ \pm 1 \},
\end{equation}
 where $\bar{g}$ denotes the image of $g$ in $O(q(S_X))$.
The group of symplectic automorphisms corresponds to
\begin{equation}
 O^+_\GAM(S_X) \cong
 \left( \CLFP{L}^1 \sqcup
  \{ \alpha \in \CLFM{L} \bigm| N \alpha = -1 \} \right)
  / \{ \pm 1 \}.
\end{equation}
(See Proposition \ref{prop:odd_pic_num}
 and Corollary \ref{cor:ortho_plus}.)
\end{corollary}

For non-zero integers $k$ and $l$,
 we define a $\Z$-subalgebra $B_{k,l}$
 of the matrix algebra $M_2(\Z)$ of size $2$ by
\begin{equation}
 B_{k,l} :=
 \{ \alpha=\begin{pmatrix} a& b \\ c& d \end{pmatrix} \in M_2(\Z)
 \bigm|
 a-d \equiv c \equiv 0 \bmod k, ~
 b\equiv 0 \bmod l \}.
\end{equation}
Set $\LAT_{k,l} := U(k) \OSUM \LL{2 l}$.
Then $\OP{Cl}^+(\LAT_{k,l})$ is isomorphic to $B_{k,l}$ as a $\Z$-algebra.
(We study the lattice $\LAT_{k,l}$ in Section \ref{sect:LAT_kl}.)
In this setting, for $\alpha \in B_{k,l}$
 we have $N \alpha = \det(\alpha)$.
Hence
\begin{equation}
 SO_\GAM(\LAT_{k,l}) \cong B_{k,l}^\times / \{ \pm 1 \}, \quad
 B_{k,l}^\times
 = \{ \alpha \in B_{k,l} \bigm|  \det(\alpha)=\pm 1 \}.
\end{equation}
By Lemma \ref{lem:condition_negative_det},
 $B^\times_{k,l} \subseteq SL_2(\Z)$
 if and only if $-1$ is not a quadratic residue modulo $k$.
In addition, by Proposition \ref{prop:CL_minus_unit},
 $O_\GAM(\LAT_{k,l})=SO_\GAM(\LAT_{k,l})$
 if and only if $\LAT_{k,l}$ does not represent $\pm 2$,
 that is, any $v \in \LAT_{k,l}$ has self-intersection number
 $\neq \pm 2$.
Therefore, we obtain:

\begin{proposition} \label{prop:auto-k3-Xn}
For $n \in \Z_{\geq 2}$, let $X_n$
 be a K3 surface whose Picard lattice is isomorphic to
 $U(n) \OSUM \LL{-2 n}$.
Then
\begin{equation}
 \Aut(X_n) \cong G_n := B_{n,n}^\times / \{ \pm I \}
 \subseteq PGL_2(\Z).
\end{equation}
Moreover, the image of the natural map
 $\Aut(X_n) \rightarrow GL(H^{2,0}(X_n)) \cong \C^\times$
 is a cyclic group of order $m=1$ or $2$,
 where $m=2$ if and only if $-1$ is a quadratic residue modulo $n$.
\end{proposition}

Set $\Gamma:=PSL_2(\Z)$ and $\Pi:=PGL_2(\Z)$.
For an integer $n\geq 1$, let $\Gamma(n)$
 denote the principal congruence subgroup of level $n$, 
 that is,
\begin{equation*} 
 \Pi(n) := \{ \alpha \in \Pi  \bigm| \alpha \equiv \pm I
 \bmod n \}, \quad
 \Gamma(n) := \Pi(n)\cap \Gamma.
\end{equation*}

\begin{example} \label{ex:wehler}
When $n=2$ in Proposition \ref{prop:auto-k3-Xn},
 we have Wehler K3 surfaces $X_2$ as described above,
 whose automorphism groups are isomorphic to
 $G_{2}=\Pi(2)$.
\end{example}

The subgroup $\Gamma_n := G_n \cap \Gamma \subseteq \Gamma$
 is a congruence subgroup 
 containing $\Gamma(n)$.
The quotient $\Gamma/\Gamma_n$
 is isomorphic to the product group
 $\prod_{q || n} PSL_2(\Z/q \Z)$,
 where $q=p^e$ runs over the prime power divisors with $q || n$,
 that is, $p$ is a prime 
 appearing exactly $e$ times in the prime factorization of $n$.
(See Section \ref{sect:lat_u_n_2n} for details.)
We have $G_1=\Pi$, 
 $[\Pi:G_2]=6$ and, for $n \geq 3$,
\begin{gather*}
 [\Pi : G_n]
 = \frac{[\Pi : \Gamma(n)]}{[G_n : \Gamma(n)]}
 = \frac{2 \cdot |\Gamma/\Gamma(n)|}{\delta_n}
 = \frac{1}{\delta_n} \times
 n^3 \prod_{p|n} \left( 1-\frac{1}{p^2} \right), \\
 \delta_n := |G_n/\Gamma(n)|
 = |\{ \bar{a} \in (\Z/n \Z)^\times \bigm|
  a^2 \equiv \pm 1 \bmod n \} / \{ \pm 1 \bmod{n} \}|,
\end{gather*}
 where the product is taken over
 the prime divisors $p$ of $n$.

\begin{example}
As a special case of Theorem \ref{prop:auto-k3-Xn},
 if $n$ is equal to a prime power $p^e$ with $e \geq 1$,
 then $\Aut(X_n) \cong G_n$ is isomorphic to $\Gamma(n)$
 or an overgroup of it of index $2$, as follows:
\begin{equation*}
 G_n =
 \begin{cases}
  \Pi(2) = \ANG{ \Gamma(2) , \SqMat{1}{0}{0}{-1} } & \quad (n=2) \\
  \Gamma(4) & \quad (n=4) \\
  \ANG{ \Gamma(2^e) ,%
   \SqMat{1+2^{e-1}}{2^e}{2^{2e-3}}{1-2^{e-1}+2^{2(e-1)}} }
   & \quad (n=2^e, ~ e \geq 3) \\
  \Gamma(p^e) & \quad (n=p^e, ~ p \equiv 3 \bmod 4) \\
  \ANG{ \Gamma(p^e) , \SqMat{a}{p^e}{(a^{2 p^{\scriptstyle e}}+1)/p^{e}}{a^{2 p^{\scriptstyle e} - 1}} }
   & \quad (n=p^e, ~ p \equiv 1 \bmod 4)
 \end{cases}
\end{equation*}
Here, in the last case,
 we chose an integer $a$ so that $a^2 \equiv -1 \bmod p^e$.
\end{example}

By a standard argument for modular groups
 (see
 Propositions \ref{prop:torsion-grp} and \ref{prop:free-group-12}),
 if a subgroup $G \subseteq \Pi$ of finite index is torsion-free,
 then $G$ is a free group of rank $\frac{1}{12}[\Pi:G]+1$.
In fact, $G_n$ is torsion-free if and only if $n \geq 3$.

\begin{example}
If $n=2^e$ with $e \geq 3$,
 we have $[\Pi:G_{2^{\scriptstyle e}}]=3 \cdot 2^{3e-3}$.
Thus
 $\Aut(X_{2^{\scriptstyle e}}) \cong G_{2^{\scriptstyle e}}$
 is a free group of rank $2^{3e-5} + 1$.
(See Example \ref{exm:2-power-gamma}.)
\end{example}

\begin{corollary}
A free group of arbitrarily large rank occurs as the automorphism group
 of a projective K3 surface with Picard number $3$.
In particular,
 the number of generators of the automorphism group of
 a projective K3 surface with Picard number $3$
 has no upper bound.
\end{corollary}

\subsection*{Structure of the paper}

We summarize the contents of this paper.
In Section \ref{sect:preliminaries}, we recall some results on
 lattices, K3 surfaces and
 principal congruence subgroups of modular groups.
In Section \ref{sect:cliff}, we prove Main Theorem.
In the next two sections, 
 we apply Main Theorem
 to lattices $U(k) \OSUM \LL{2l}$,
 and as a special case, $U(n) \OSUM \LL{2n}$.
In Section \ref{sect:salem},
 we compute Salem polynomials of automorphisms of
 Wehler K3 surfaces $X_2$.
The last two sections are appendices.

\subsection*{Acknowledgement}
We use Maxima \cite{maxima} for some computations in the paper.
We thank the developers of this excellent software.

\section{Preliminaries} \label{sect:preliminaries}

\subsection{Convention} \label{subsect:convention}
Throughout the paper,
 we use the same notation for elements (matrices) in
 general linear groups and their images in projective linear groups,
 for simplicity.
The product symbol $\prod_{p | n}$
 is taken over the positive prime divisors of $n$.
For example,
 $\prod_{p | n} p$ for a non-zero integer $n$ is
 the radical of $|n|$.
Tensor products over $\Z$ are denoted by subscript.
For example,
 a quarternion algebra $B$
 is mostly defined over $\Z$ in the paper
 and its extension of base ring by e.g.\ $\Q$ is denoted by
 $B\OverQ$.
(In addition, the unit group $(B\OverQ)^\times$
 of $B\OverQ$ is sometimes denoted by
 $B\OverQ^\times$ for simplicity.)
We use the following convention for lattices and K3 surfaces
 (see Sections \ref{subsect:lattices} and \ref{subsect:k3-basics}
 for details).
Lattices are always non-degenerate,
 unless otherwise stated.
K3 surfaces are defined over $\C$,
 which may not be projective.
However, the Picard lattice $S_X$
 of a K3 surface $X$
 is always supposed to be non-degenerate.
In addition, we assume that
 $X$ contains no $(-2)$-curves, that is,
 $\ROOT(S_X) = \varnothing$ (see (\ref{eq:root_L})).

\subsection{Lattices} \label{subsect:lattices}

A lattice
 is a free $\Z$-module $L$
 of finite rank equipped with a symmetric bilinear form
 $\langle ~,~ \rangle \colon L\times L\rightarrow \Z$.
If $\langle x,x\rangle \in 2\Z$ for any $x\in L$,
 a lattice $L$ is said to be even.
We sometimes fix a $\Z$-basis $(e_i)$ of $L$
 and identify the lattice $L$
 with its intersection matrix $Q_L := (\ANG{e_i,e_j})_{ij}$
 under this basis.
In particular,
 a lattice of rank one
 whose generator has self-intersection number $a$
 is represented as $\LL{a}$.
The discriminant $\disc(L)$ of $L$ is defined as $\det(Q_L)$,
 which is independent of the choice of a basis.
A lattice $L$ is called non-degenerate if
 $\disc(L)\neq 0$ and unimodular if $\disc(L)=\pm1$.
For a non-degenerate lattice $L$,
 the signature of $L$ is defined as $(s_+,s_-)$,
 where $s_{+}$ (resp.\ $s_-$)
 denotes the number of the positive (resp.\ negative) eigenvalues of
 $Q_L$.
A (self-)isometry of $L$ is an automorphism of the $\Z$-module $L$
 preserving the bilinear form.
The isometries of $L$ form the orthogonal group $O(L)$ of $L$.
The special orthogonal group $SO(L)$ of $L$
 consists of the isometries of $L$ of determinant $1$.
Let $K$ be a sublattice of a lattice $L$,
 that is, $K$ is a $\Z$-submodule of $L$
 equipped with the restriction of the bilinear form of $L$ to $K$.
If $L/K$ is torsion-free as a $\Z$-module,
 $K$ is said to be primitive.
For a ring $\mathbb{S}$ containing $\Z$,
 we equip $L_\mathbb{S} = L \otimes \mathbb{S}$
 with the $\mathbb{S}$-linear extension 
 of the bilinear form $\ANG{~,~}$ on $L$,
 denoted by the same symbol.
We set
\begin{equation} \label{eq:root_L}
 \ROOT(L) := \{ v \in L \bigm| \ANG{v,v}=-2 \}.
\end{equation}
The hyperbolic lattice is denoted by $U$:
\begin{equation}
 U = \begin{pmatrix}  0&1\\ 1&0 \end{pmatrix}.
\end{equation}

{\bf Orthogonal subgroups $O^+(L)$ and $SO^+(L)$.}
Let $L$ be a lattice with signature $(k,l)$.
In our convention,
 $O^+(L)$ is defined as the group
 consisting of $g \in O(L)$ such that
 the signs of $\det(g)^{m+k}$ and $\theta(g)$ coincide,
 where $m := (k+1)(l+1)$ and
 $\theta(g)$ denotes the spinor norm of $g$.
(See Section \ref{sect:def_cliff} for spinor norm.)
Alternatively, $O^+(L)$ is defined as follows.
Suppose that $k$ is odd; or, both $k$ and $l$ are even.
Then $m+k$ is odd.
Define the positive Grassmannian $Gr^+(k,L_\R)$ by
\begin{equation*}
 Gr^+(k,L_\R) :=
 \{ V \subseteq L_\R \bigm| V \cong \R^k, ~
 \ANG{~,~}|_V \gg 0 \},
\end{equation*}
 that is, $Gr^+(k,L_\R)$
 consists of $k$-dimensional subspaces $V \subseteq L_\R$
 such that $\ANG{~,~}$ is positive definite on $V$.
Then we define the set of orientation elements for $Gr^+(k,L_\R)$ by
\begin{equation*}
 \Omega :=
 \{ \omega \in \wedge^k L_\R \bigm|
 \omega \in \wedge^k V ~ (\exists V \in Gr^+(L_\R)), ~
 \omega \neq \mathbf{0} \}
 = \Omega^+ \sqcup (-\Omega^+),
\end{equation*}
 where $\Omega^+$ and $-\Omega^+$
 are connected components of $\Omega$.
Now set
\begin{equation} \label{O_and_SO}
 O^+(L) := \{ g \in O(L) \bigm|
 g(\CONE) = 
 \CONE \}, \quad
 SO^+(L) := O^+(L)\cap SO(L).
\end{equation}
When $k$ is even and $l$ is odd,
 we use a similar construction
 for the negative Grassmannian $Gr^{-}(l,L_\R)$.
In this convention, if either $k$ or $l$ is odd,
 we have
\begin{equation}
 O(L) = O^+(L) \times \{ \pm \id_L \}, \quad
 O^+(L) =
 \{ \OP{sgn}(\theta(g)) \cdot g \bigm| g \in SO(L)\},
\end{equation}
 where $\OP{sgn}(\theta(g)) \in \{ \pm 1 \}$
 denotes the sign of $\theta(g)$.

{\bf Discriminant group and form.}
For a lattice $L$,
 the discriminant group $A(L)$ of $L$ is defined by
\begin{equation} 
 A(L):=L^\vee/L, \quad
 L^\vee :=
 \{ x\in L\OverQ \bigm|
 \langle x,y \rangle \in \Z ~ (\forall y\in L) \}.
\end{equation}
We have $|A(L)|=|\disc(L)|$.
When $L$ is even, we define the discriminant form
 $q(L)$ of $L$ by
\begin{equation}
 q(L) \colon A(L)\rightarrow \Q/2 \Z; \quad
 x \bmod L \mapsto \ANG{x,x} \bmod 2 \Z.
\end{equation}
The group of automorphisms of $A(L)$
 preserving $q(L)$
 is denoted by $O(q(L))$.

{\bf Discriminant kernel and image.}
The discriminant kernel $G_\GAM$
 for a subgroup $G \subseteq O(L)$
 is defined by
\begin{equation} \label{eq:discriminant_kernel}
 G_\GAM :=
 \{ g \in G \bigm| \bar{g}=\id_{A(L)} \}
 = \OP{Ker}(G \rightarrow \Aut(A(L))),
\end{equation}
 where $\bar{g}$ denotes the automorphism of $A(L)$ induced by $g$.
In particular, when $L$ is even, we have
\begin{equation} \label{eq:def_OA}
 O_\GAM(L) =
 \OP{Ker}(O(L) \rightarrow O(q(L))).
\end{equation}
The discriminant image $G^\GAMu$ of $G$ is defined by
\begin{equation}
 G^\GAMu := \OP{Image}(G \rightarrow O(q(L)))
 \cong G/G_\GAM.
\end{equation}
We have a natural short exact sequence
\begin{equation}
 1 \rightarrow G_\GAM \rightarrow G 
 \rightarrow G^\GAMu \rightarrow 1.
\end{equation}

{\bf $\Q$-valued lattices.}
It is sometimes convenient to consider a $\Q$-valued lattice $L$,
 that is, we relax the definition of lattice so that
 $Q_L$ may have entries in $\Q$.
The concepts of $\disc(L)$ and $L^\vee$ 
 are valid for $\Q$-valued lattices as well.
For $\kappa \in \Q$,
 we mean by $L(\kappa)$
 the ($\Q$-valued) lattice defined as the $\Z$-module $L$
 equipped with $\kappa$ times the bilinear form of $L$,
 that is, $Q_{L(\kappa)}=\kappa \cdot Q_L$.


\begin{lemma} \label{lem:lattice_power}
Let $s$ be a non-zero integer and
 $L$ an even lattice of rank $n$.
We consider $L$ as
 $\Z^n$ (as column vectors) with intersection matrix $Q_L$
 and use the following identification:
\begin{equation*}
 O(L)=O(L(s))= \{ g\in GL_n(\Z) \bigm| g^T \cdot Q_L \cdot g=Q_L \}.
\end{equation*}
Then, for $g \in O(L)$, we have the following:
\begin{itemize}
\item[(1)]
 $g \in O_\GAM(L)$ if and only if
 $(g-I_n) \cdot Q_L^{-1}$ is an integer matrix, where
 $I_n$ is the identity matrix of size $n$;
\item[(2)]
 if $\frac{1}{s} Q_L$ is an integer matrix and $g \in O_\GAM(L)$,
 then $g^s \in O_\GAM(L(s))$.
\end{itemize}
\end{lemma}
\begin{proof}
The assertion (1) follows from the fact that $L^\vee$
 is generated by the columns of $Q_L^{-1}$.
Under the assumption of (2),
 we find that
 $[(g-I_n) \cdot Q_L^{-1}] \cdot \frac{1}{s} Q_L$
 is an integer matrix by (1),
 which implies $g \equiv I_n \bmod s$.
Thus
\begin{equation*}
 (g^s-I_n) \cdot Q_{L(s)}^{-1}
 =[ \frac{1}{s}(g^{s-1}+g^{s-2}+\cdots+I_n) ]
 \cdot [ (g-I_n) \cdot Q_L^{-1} ]
\end{equation*}
 is an integer matrix, and hence $g^s \in O_\GAM(L(s))$.
\end{proof}

As a consequence of
 Strong Approximation Theorem for quadratic forms
 (see e.g.\ \cite{O}),
 Proposition \ref{prop:strong-approx} (below)
 describes discriminant images $G^\GAMu$
 in terms of the spinor norm map $\theta$.
(See Section \ref{sect:def_cliff} for spinor norm.)
Before stating Proposition \ref{prop:strong-approx},
 we introduce some objects.
Define
\begin{equation}
 f := \det \times \, \theta \colon
 O(L)\rightarrow
 \{ \pm 1 \} \times \SQ{\Q}, \quad
 \Theta(L) := \OP{Ker}(f).
\end{equation}
For each prime $p$, set $L_p := L \otimes \Z_p$ and define
\begin{equation}
 f_p := \det \times \, \theta_p \colon
 O(L_p)\rightarrow
 \{ \pm 1 \} \times \SQ{\Q_p}, \quad
 \Theta(L_p) := \OP{Ker}(f_p).
\end{equation}
Here the spinor norm
 $\theta_p(g_p)\in\Q_p^\times/(\Q_p^\times)^2$
 of $g_p\in O(L_{p})$
 is defined in a similar way.
The codomain of $f_p$
 is isomorphic to a $2$-elementary abelian group
 $(\Z/2 \Z)^{r}$
 with $r=4$ if $p=2$ and $r=3$ otherwise.
Discriminant kernels for $L_p$
 are defined and represented similarly:
\begin{equation}
 O_\GAM(L_p)=\OP{Ker} \bigl( O_\GAM(L_p)\rightarrow
 O(q(L_p)) \bigr), \quad
 \Theta_\GAM(L_p) = \Theta(L_p) \cap O_\GAM(L_p).
\end{equation}
Since $q(L)$ is uniquely written
 as the orthogonal sum of the $q(L_p)$,
 we have a natural isomorphism
\begin{equation}
 O(q(L))=\prod_{p|d} O(q(L_p))
 \cong \prod_{p|d}
 \frac{O(L_p)}{O_\GAM(L_p)}, 
 \quad d := \disc(L).
\end{equation}
Here we use the fact that
 the natural map $O(L_p) \rightarrow O(q(L_p))$
 is surjective in general \cite[Corollary 1.9.6]{N}.

Applying Strong Approximation Theorem for quadratic forms,
 one can determine discriminant images.
The following is a consequence of the results in
 \cite[Chapter VIII]{mm}.

\begin{proposition}[cf.\ {\cite[Theorem 5.1]{mm}}]
\label{prop:strong-approx}
In the same setting as above,
 let $L$ be an indefinite even lattice of rank $\geq 3$.
Let $G$ be a subgroup of $O(L)$ containing $\Theta(L)$.
(For example, one can take $G=SO^+(L)$.)
We have the following natural commutative diagram:
\begin{equation} \label{eq:com-diagram}
\begin{array}{ccccccccc}
 1 & \rightarrow & \Theta(L) & \rightarrow & G 
 & {\displaystyle \mathop{\rightarrow}^{\displaystyle f}}
 & f(G) 
 & \rightarrow & 1
 \\
 && \ \ \, \downarrow \alpha && \ \ \, \downarrow \beta &
 & \ \ \, \downarrow \gamma &&
 \\
 1 & \rightarrow
 & \displaystyle \prod_{p|d}
 \frac{\Theta(L_p)}{\Theta_\GAM(L_p)}
 & \rightarrow
 & \displaystyle \prod_{p|d}
 \frac{O(L_p)}{O_\GAM(L_p)}
 & \rightarrow
 & \displaystyle \prod_{p|d}
 \frac{f_p(O(L_p))}{f_p(O_\GAM(L_p))}
 & \rightarrow & 1
\end{array}
\end{equation}
Here each product is taken over the (positive) prime divisors of
 $d := \disc(L)$.
In this diagram, 
 (1) the map $\alpha$ is surjective; and
 (2) the both rows 
 are exact.
Therefore, by diagram chasing, one can check that
 the discriminant image $G^\GAMu \cong \OP{Image}(\beta)$
 is a normal subgroup of $O(q(L))$, that is,
\begin{equation}
 G^\GAMu = \OP{Image}(G \rightarrow O(q(L)))
 \unlhd O(q(L)).
\end{equation}
Moreover, we have a natural isomorphism
\begin{equation}
 \OP{Coker}(G \rightarrow O(q(L))) =
 {O(q(L))} / {G^\GAMu}
 \cong \OP{Coker}(\beta)
 \cong \OP{Coker}(\gamma).
\end{equation}
(By the construction,
 $\OP{Coker}(\gamma)$ is a $2$-elementary abelian group.)
In particular, the surjectivity of $\gamma$ implies
 $G^\GAMu=O(q(L))$.
\end{proposition}
\begin{proof}
By Strong Approximation Theorem for quadratic forms
 (see e.g.\ \cite{O}), 
 the map $\alpha$ is surjective.
In order to verify the exactness of
 the second row of (\ref{eq:com-diagram}),
 we use the following commutative diagram:
\begin{equation*}
\begin{array}{ccccccccc}
 1 & \rightarrow & \Theta_\GAM(L_p) & \rightarrow & O_\GAM(L_p)
 & \rightarrow
 & f_p(O_\GAM(L_p))
 & \rightarrow & 1
 \\
 && ~ \downarrow && ~ \downarrow &
 & ~ \downarrow &&
 \\
 1 & \rightarrow & \Theta(L_p) & \rightarrow & O(L_p)
 & \rightarrow
 & f_p(O(L_p))
 & \rightarrow & 1
\end{array}
\end{equation*}
The both rows of this diagram are exact,
 which implies the required exactness
 (by applying the nine lemma).
\end{proof}

\begin{remark}
In Proposition \ref{prop:strong-approx},
 the assumption of $L$ being indefinite is essential.
For example, if $L$ is isomorphic to the root lattice $D_n$
 of rank $n$, where $n \equiv 4 \bmod{8}$ and $n \geq 12$,
 then it is known that $O^\GAMu(L)$ 
 is a cyclic group of order $2$.
On the other hand, we have
\begin{equation*}
 A(L) \cong (\Z/2 \Z)^{\oplus 2}, \quad
 q(L) \cong \SqMat{1}{1/2}{1/2}{1}, \quad
 O(q(L)) \cong S_3.
\end{equation*}
Here $S_3$ denotes the symmetric group of degree $3$.
\end{remark}

\subsection{K3 surfaces} \label{subsect:k3-basics}

A compact complex surface $X$
 is called a K3 surface if it is simply connected and
 has a nowhere vanishing holomorphic $2$-from $\omega_X$
 (see e.g.\ \cite{BHPV} for details).
We consider the second integral cohomology $H^2(X,\Z)$
 with the cup product as a lattice.
It is known that $H^2(X,\Z)$ is an even unimodular lattice
 of signature $(3,19)$.
Such a lattice is unique up to isomorphism (of abstract lattices)
 and is called the K3 lattice.
We fix such a lattice and write it as $\LatK$.
The Picard lattice $S_X$
 and transcendental lattice $T_X$ of $X$ are defined by
\begin{align}
 S_X & {} := \{ x \in H^2(X,\Z) \bigm| \langle x,\omega_X \rangle=0 \}, \\
 T_X & {} := (S_X)^\bot
  = \{ y \in H^2(X,\Z) \bigm|
  \langle x,y \rangle=0~(\forall x \in S_X) \}.
\end{align}
Here $\omega_X$ is considered as an element in $H^2(X,\C)$
 and the bilinear form on $H^2(X,\Z)$ is extended to that on
 $H^2(X,\C)$ linearly.
The Picard group
 of $X$
 is naturally isomorphic to $S_X$.
It is known that $X$ is projective if and only if
 $S_X$ is non-degenerate and has signature $(1,\rho-1)$,
 where $\rho=\rank\, S_X$ is the Picard number of $X$.

Let $X$ be a projective K3 surface.
Since $H^2(X,\Z)$ is unimodular, 
 we have the following natural isomorphisms:
\begin{equation} \label{eq:Sx-Tx}
 A(S_X)\cong A(T_X)\cong H^2(X,\Z)/(S_X \OSUM T_X)
\end{equation}
 (see \cite{N} for details).
By the global Torelli theorem for K3 surfaces \cite{BR,PS},
 the following map is injective:
\begin{equation} \label{EQ_map_torelli}
 \Aut(X) \rightarrow O(S_X) \times O(T_X); \quad
 \varphi \mapsto
 (g,h):=(\varphi^*|_{S_X},\varphi^*|_{T_X}).
\end{equation}
Moreover,
 $(g,h) \in O(S_X) \times O(T_X)$
 is the image of some $\varphi \in \Aut(X)$
 by the map (\ref{EQ_map_torelli}) if and only if
 (1) the linear extension of $g$ (resp.\ $h$)
 preserves the ample cone 
 of $X$ (resp.\ $H^{2,0}(X) = \C\omega_X$)
 and (2) the actions of $g$ and $h$ on $A(S_X) \cong A(T_X)$
 (via (\ref{eq:Sx-Tx})) coincide.

An automorphism $\varphi$ of a K3 surface $X$
 is said to be {\it symplectic}
 if $\varphi^* \omega_X=\omega_X$,
 and {\it anti-symplectic}
 if $\varphi^* \omega_X=-\omega_X$.
These conditions are equivalent to
 $\varphi^*|_{T_X}=\id_{T_X}$ and $\varphi^*|_{T_X}=-\id_{T_X}$,
 respectively
 \cite[Theorem 3.1]{Nfin}.

\begin{remark}
It may be worth mentioning the following application of
 Proposition \ref{prop:strong-approx}
 to automorphisms of K3 surfaces.
Let $X$ be a projective K3 surface with Picard number $\geq 3$
 and $\ROOT(S_X)=\varnothing$.
Suppose that $h \in O(T_X)$ preserves $H^{2,0}(X)$.
By Proposition \ref{prop:strong-approx},
 the image of $h^2$ in $O(q(T_X))$ 
 coincides with that of some $g \in \Theta(S_X)$
 under the natural isomorphism $A(S_X) \cong A(T_X)$.
By the global Torelli theorem for K3 surfaces,
 $g \oplus h^2$ is induced by an automorphism of $X$.
For a similar statement for Picard number $2$,
 see \cite{HKL}.
\end{remark}

\begin{theorem} \label{THM_Nikulin2}
Let $\rho$ be a positive integer with $\rho\leq 10$.
Then any even lattice $S$
 of signature $(1,\rho-1)$ or $(0,\rho)$
 occurs as the Picard lattice $S_X$ of some K3 surface $X$.
\end{theorem}
\begin{proof}
The K3 lattice $\LatK$
 contains a primitive sublattice isomorphic to $S$
 \cite[Theorem 1.14.4]{N}.
(In fact, such a sublattice is unique up to $O(\LatK)$.)
The surjectivity of the period map for K3 surfaces \cite{Tod}
 implies the existence of a K3 surface $X$
 and an isomorphism $H^2(X,\Z) \cong \LatK$
 mapping $S_X$ to $S$.
\end{proof}

\begin{proposition} \label{prop:odd_pic_num}
Let $X$ be a projective K3 surface with odd Picard number $\rho$.
Then any automorphism of $X$ is symplectic or anti-symplectic.
Moreover,
 if $\rho \geq 3$ and $\ROOT(S_X) = \varnothing$,
 we have a (group) isomorphism
\begin{equation} \label{eq:mapping_SX}
 \Aut(X) \cong
  \{ g \in O^+(S_X) \bigm| \bar{g} = \pm \id
  \ \, \text{\rm in} \ \, O(q(S_X)) \} =: {} H; \quad
 \varphi \mapsto (\varphi^{-1})^*|_{S_X}.
\end{equation}
In particular, any non-trivial automorphism of $X$
 acts on $S_X$ non-trivially
 (equivalently, $S_X$ is not a $2$-elementary abelian group).
The group of symplectic automorphisms corresponds to
 $O^+_\GAM(S_X)$.
We also have a natural isomorphism $H \cong O_\GAM(S_X)$.
\end{proposition}
\begin{proof}
Since the Picard number of $X$ is odd,
 any $\varphi \in \Aut(X)$ is symplectic or anti-symplectic
 \cite[Theorem 3.1]{Nfin}.
Assume $\rho \geq 3$ and $\ROOT(S_X) = \varnothing$.
If $\varphi \neq \id$ acts on $S_X$ trivially,
 then $\varphi$ is an anti-symplectic involution and
 we should have $\ROOT(S_X) \neq \varnothing$
 by the classification of anti-symplectic involutions on K3 surfaces
 \cite{Ninvo}.
(In particular, $\id \neq - \id$ in $O(q(S_X))$.)
Hence the map (\ref{eq:mapping_SX}) is an isomorphism
 by the global Torelli theorem for K3 surfaces.
For $g \in H$,
 let $\varepsilon_g= 1$ or $-1$
 according to whether $\bar{g}=\id$ or $-\id$.
Then the correspondence
 $g \leftrightarrow \varepsilon_g \cdot g$
 gives an isomorphism $H \cong O_\GAM(S_X)$.
\end{proof}

\begin{proposition}
Let $X$ be a projective K3 surface with odd Picard number $\leq 7$.
If $g \in Aut(X)$ has finite order,
 then $g$ is an anti-symplectic involution.
\end{proposition}
\begin{proof}
In general, if a K3 surface admits
 a non-trivial symplectic automorphism of finite order,
 its Picard number should be at least $8$
 (see \cite{Nfin}).
The assertion of the proposition follows from
 Proposition \ref{prop:odd_pic_num}.
\end{proof}

\subsection{Congruence subgroups} \label{subsect:cong}

We recall some basic facts on congruence subgroups
 (see e.g.\ \cite{JT,K}).
Let $\Pi$ denote the extended modular group,
 and $\Gamma$ the modular group, a subgroup of index $2$ in $\Pi$:
\begin{equation}
 \Pi := PGL_2(\Z) = GL_2(\Z)/\{\pm I \}, \quad
 \Gamma := PSL_2(\Z) = SL_2(\Z)/\{\pm I \}.
\end{equation}
For each integer $n \geq 1$,
 the principal congruence subgroup $\Gamma(n)$ of level $n$,
 which is a normal subgroup of finite index in $\Gamma$,
 is defined by
\begin{equation}
\Pi(n) := \{ \alpha \in \Pi  \bigm| \alpha \equiv \pm I
 \bmod n \}, \quad
 \Gamma(n) := \Pi(n)\cap \Gamma.
\end{equation}
It follows that
 $\Gamma/\Gamma(n) \cong SL(2,\Z/n \Z) / \{ \pm I \bmod{n} \}$.
We have
\begin{equation}
 \Pi(1)=\Pi, \quad \Gamma(1)=\Gamma; \quad
 [\Pi(2):\Gamma(2)]=2, \quad
 [\Gamma:\Gamma(2)]=6;
\end{equation} 
 and, for $n\geq 3$,
\begin{equation}
 \Pi(n)=\Gamma(n), \quad
 [\Gamma:\Gamma(n)]=\frac{n^3}{2}\prod_{p|n}(1-\frac{1}{p^2}).
\end{equation} 
A subgroup $G$ of $\Gamma$ is called a congruence subgroup
 if there exists $n$ such that it contains $\Gamma(n)$.
The level of $G$ is then the smallest such $n$.

\begin{proposition}[{cf.\ \cite[Lemma 3.2]{SIKpower}}] \label{prop:torsion-grp}
Suppose that $P \unlhd \Pi$ has finite index
 and contains a torsion element
 (i.e.\ a non-trivial element of finite order).
Then $P$ contains one of the following three subgroups
 (of indices $6$, $4$ and $6$, respectively):
\begin{equation*}
 \Pi(2)=\ANG{\Gamma(2),(\SqMatSmall{1}{0}{0}{-1})}, \quad
 \ANG{\Gamma(2),(\SqMatSmall{1}{1}{-1}{0})}
 \quad \text{and} \quad
 \ANG{\Gamma(3),(\SqMatSmall{0}{1}{-1}{0}),(\SqMatSmall{1}{2}{-1}{-1})}.
\end{equation*}
(Here the second and third subgroups are equal to $\Gamma^2$ and $\Gamma^3$,
 respectively,
 where $\Gamma^m$ denotes the $m$th power subgroup of $\Gamma$.)
\end{proposition}

In order to study the structure of congruence subgroups,
 we apply the following famous theorems in group theory.

\begin{theorem}[Nielsen--Schreier Theorem, see e.g.\ {\cite{R}}] \label{Nielsen-Schreier}
Every subgroup $F'$ of a free group $F$ is free.
Moreover, if both the rank $r$ of $F$
 and the index $e := [F:F']$ are finite,
 then the rank $r'$ of $F'$ is also finite and satisfies
 $e(r-1)=r'-1$.
\end{theorem}

\begin{theorem}[Kurosh's Theorem \cite{Ma}] \label{Kurosh's theorem}
Let $G=A\ast B$ be the free product of groups $A$ and $B$,
 and let $H \subseteq G$ be a subgroup of $G$.
Then there exist a family $(A_i)_{i\in I}$ of subgroups $A_i \subseteq A$,
 a family $(B_j)_{j\in J}$ of subgroups $B_j \subseteq B$,
 families $(g_i)_{i\in I}$ and $(f_j)_{j\in J}$ of elements of $G$,
 and a subset $S\subseteq G$ such that 
\begin{equation*}
 H=F(S)\ast(\ast_{i\in I} \, g_iA_ig_i^{-1})\ast (\ast_{j\in J} \, f_jB_jf_j^{-1}),
\end{equation*}
 where $F(S)$ is the free group with free basis $S$.
\end{theorem}

Since $\Gamma = PSL_2(\Z)$
 is isomorphic to the free product $C_2 * C_3$,
 it follows from Kurosh's theorem that any subgroup of
 $\Gamma$ is the free product of a certain number of copies of
 $C_2$, $C_3$ and the infinite cyclic group $C_\infty$.
For example, for $n \geq 2$,
 the principal congruence subgroup $\Gamma(n)$
 is torsion-free (that is, it contains no torsion elements)
 by Proposition \ref{prop:torsion-grp}.
Hence $\Gamma(n)$ is a free group.
Moreover, the rank of $\Gamma(n)$ as a free group
 is equal to $\frac{1}{6}[\Gamma:\Gamma(n)]+1$
 (see below).


One can study subgroups of $\Pi$ of finite index
 in a topological way, as follows.
Let $\MF{H}$ denote the upper-half plane.
We identify $\MF{H}$ with the following quotient space:
\begin{equation*}
 \{ \tau \in \C \bigm| \OP{Re}(\tau) \neq 0 \}
 / \ANG{\rho},
\end{equation*}
 where $\rho \colon \tau \mapsto \bar{\tau}$
 is the involution defined by complex conjugate.
Then $PGL_2(\R)$ acts on $\MF{H}$ by linear fraction, as follows:
\begin{equation*}
 g \cdot \tau = \frac{a \tau+b}{c \tau+d}\,, \qquad
 g=\SqMat{a}{b}{c}{d} \in PGL_2(\R).
\end{equation*}
The action of $g \in PGL_2(\R)$ on $\MF{H}$ is holomorphic if
 $g \in PSL_2(\R)$ and anti-holomorphic otherwise.
It is known that the quotient $\MF{H}/G$ of $\MF{H}$
 by a subgroup $G \subseteq \Gamma$
 of finite index becomes a (compact) complex curve
 with a finite number ($\geq 1$) of punctures, or ``cusps.''
(Precisely, if $G$ contains torsion elements,
 $\MF{H}/G$ is considered as an orbi-curve.)

We recall the following basic fact on topological surfaces:

\begin{theorem}
Let $Y$ be a closed topological surface with $r$ punctures.
Then we have the following two cases:
\begin{itemize}
\item[(1)]
 $Y$ is an orientable surface of genus $g \geq 0$, and
\begin{equation*}
 \pi_1(Y) \cong \ANG{a_i,b_i ~ (1 \leq i \leq g), ~ t_j ~ (1 \leq j \leq r) \bigm|
 [a_1,b_1] \cdots [a_g,b_g] t_1 \cdots t_r=1}\text{;}
\end{equation*}
\item[(2)]
 $Y$ is non-orientable and homeomorphic to a connected sum of
 $h \geq 1$ copies of the real projective plane, and
\begin{equation*}
 \pi_1(Y) \cong \ANG{c_i ~ (1 \leq i \leq h), ~ t_j ~ (1 \leq j \leq r) \bigm|
 c_1^2 \cdots c_h^2 t_1 \cdots t_r=1}.
\end{equation*}
\end{itemize}
Here $[a,b] := a b a^{-1} b^{-1}$.
In particular, if $r \geq 1$,
 then the fundamental group $\pi_1(Y)$
 is isomorphic to a free group of rank $2g+r-1$ in Case (1)
 and of rank $h+r-1$ in Case (2).
\end{theorem}

This fact, combined with Nielsen--Schreier formula,
 implies the following:

\begin{proposition} \label{prop:free-group-12}
Let $P$ be a subgroup of $\Pi$ of finite index.
Suppose that $P$ is torsion-free.
Then $P$ is isomorphic to a free group of rank
 $\frac{1}{12}[\Pi:P]+1$.
\end{proposition}
\begin{proof}
Since $P$ contains no torsion elements,
 the quotient map $\MF{H} \rightarrow \MF{H}/P$ is non-ramified
 and $P=\pi_1(\MF{H}/P)$ is a free group.
In the case $P \subseteq \Gamma$,
 the quotient $\mathfrak{H}/P$
 is a complex curve of genus $g$ with (say) $r$ punctures
 and we have $\frac{1}{6} [\Gamma:P]=2(g-1)+r$,
 as is well known in the theory of elliptic modular forms.
Hence the assertion holds in this case.
In the case $P \not\subseteq \Gamma$,
 we apply Nielsen--Schreier Theorem to $P \cap \Gamma \subset P$
 and obtain
\begin{equation*}
 \rank\, P - 1
 =\frac{1}{2} (\rank\, P \cap \Gamma - 1).
\end{equation*}
Thus the assertion for $P$ follows from 
 the assertion for $P \cap \Gamma \subseteq \Gamma$.
\end{proof}

\section{Proof of Main Theorem} \label{sect:cliff}

\subsection{Statement of Main Theorem}

Let $L$ be an 
 even lattice of rank $3$.
In this section, we prove Theorem \ref{thm:MAIN} below,
 which describes the structure of the orthogonal subgroups
 $O_\GAM(L)$ and so on.
We fix the following notations
(see Sections \ref{sect:def_cliff}--\ref{sect:def_quaternion}
 for the definitions):
\begin{align*}
 \CLF{L} = \CLFP{L} \oplus \CLFM{L} &:
  \text{the Clifford algebra of $L$ over $\Z$} \\
 \CLFP{L} =: B, ~ \CLFM{L} &:
  \text{the even and odd parts of $\CLF{L}$, respectively}
\end{align*}
We consider $B$ as a quaternion algebra;
 its reduced trace and norm are denoted 
 by $\OP{Tr}(~)$ and $\OP{Nr}(~)$, respectively.
The lattice $L$ is naturally contained in
 the odd part $\CLFM{L}$ as a free $\Z$-submodule
 and generate $\CLF{L}$ as a $\Z$-algebra.
Set
\begin{equation}
 \CLFPM{L} := \CLFP{L} \cup \CLFM{L}.
\end{equation}
For any $\alpha \in \CLFPM{L}$, we have
 (under the assumption that $L$ has rank $3$)
\begin{equation}
 N \alpha := \alpha \alpha^* \in \Z,
\end{equation}
 where we use the involution $*$ on $\CLF{L}$
 as in Definition \ref{def:spinor_map}.
If $\alpha \in B=\CLFP{L}$,
 we have $N \alpha = \OP{Nr}(\alpha)$.
For a subset $S \subseteq \CLFPM{L}$, define
\begin{equation}
  S^\times := S \cap \CLF{L}^\times
  = \{ \alpha \in S \bigm|
  N \alpha = \pm 1 \}.
\end{equation}
Throughout this section,
 we often use the element $E \in \CLFM{L}_\Q$
 (with an ambiguity of sign) satisfying
\begin{equation}
 Ex = xE ~ (\forall x \in \CLF{L}), ~
 E^* = -E, ~
 E^2 = \frac{1}{8} \disc(L).
\end{equation}
See Lemma \ref{lem:v_dot_E} for details.

\begin{theorem}[Main Theorem] \label{thm:MAIN}
Let $L$ be a (non-degenerate) even lattice of rank $3$.
For $\alpha \in \CLFPM{L}^\times$ and $v \in L$,
 we have $\alpha v \alpha^{-1} \in L$.
(Here the product $\alpha v \alpha^{-1}$ is taken in $\CLF{L}$.)
Moreover, we have an isomorphism
\begin{equation}
 \Gamma^\times(\CLFPM{L}) := \CLFPM{L}^\times / \{ \pm 1 \}
 \CONGa 
 O_\GAM(L),
\end{equation}
 which is induced by
\begin{equation}
 \alpha \mapsto h_\alpha :=
 (v \mapsto \varepsilon_\alpha \cdot \alpha v \alpha^{-1}),
 \qquad
 \varepsilon_\alpha =
 \begin{cases}
   1 & \text{if} \quad \alpha \in \CLFP{L}, \\
  -1 & \text{if} \quad \alpha \in \CLFM{L}.
 \end{cases}
\end{equation}
Here the discriminant kernel $O_\GAM(L)$
 is defined as the kernel of the natural map
 $O(L) \rightarrow O(q(L))$.
We have $\det(h_\alpha) = \varepsilon_\alpha$.
\end{theorem}


\begin{corollary}[cf.\ Proposition {\ref{prop:odd_pic_num}}]
\label{cor:ortho_plus}
In the same setting as above,
 we assume furthermore 
 that $A(L)$ is not a $2$-elementary abelian group
 (that is, $\id \neq - \id$ in $O(q(L))$).
Then we have an isomorphism
\begin{equation} \label{eq:o_plus_isom}
 \Gamma^\times(\CLFPM{L}) \cong O_\GAM(L) \CONGa
 \{ g \in O^+(L) \bigm| \bar{g} = \pm \id_{A(L)} \}; \quad
 \pm \alpha \mapsto \varphi_\alpha,
\end{equation}
 where $\bar{g}$ denotes the image of $g$ in $O(q(L))$
 and $\varphi_\alpha$ is defined by
\begin{equation}
 \varphi_\alpha(v)
 = \varepsilon_\alpha \cdot (N \alpha) \cdot h_\alpha(v)
 = (N \alpha) \cdot \alpha v \alpha^{-1} = \alpha v \alpha^*.
\end{equation}
Moreover, we have
\begin{equation}
 \det(\varphi_\alpha) = N \alpha, \quad
 \bar{\varphi}_\alpha =
 \varepsilon_\alpha \cdot (N \alpha) \cdot \id_{A(L)}.
\end{equation}
\end{corollary}
\begin{proof}
We have $g_\alpha := (v \mapsto \alpha v \alpha^{-1}) \in SO(L)$.
Hence $(N \alpha) \cdot g_\alpha \in O^+(L)$.
See Section \ref{subsect:lattices}
 for the orthogonal subgroup $O^+(L)$
 and Definition \ref{def:spinor_map} for spinor norm.
(If $A(L)$ is a $2$-elementary abelian group,
 then $- \id_L \in O_\GAM(L)$ and
 the map (\ref{eq:o_plus_isom}) is a $2$-to-$1$ map.)
\end{proof}

\begin{corollary} \label{cor:clifford_plus}
We have isomorphisms
\begin{equation}
 SO_\GAM(L) \cong \Gamma^\times(B) := B^\times/\{ \pm 1 \}, \quad
 SO^+_\GAM(L) \cong \Gamma^1(B) := B^1/\{ \pm 1 \},
\end{equation}
 which is induced by
 $\alpha \mapsto g_\alpha := (v \mapsto \alpha v \alpha^{-1})$.
Here
\begin{equation}
 B^1 := \{ \alpha \in B \bigm| \OP{Nr}(\alpha)=1 \}
 \subset
 B^\times = \{ \alpha \in B \bigm| \OP{Nr}(\alpha)= \pm 1 \}.
\end{equation}
(Recall that $\OP{Nr}(\alpha)=N \alpha$ for $\alpha \in B$.)
\end{corollary}

\begin{remark}
In Corollary \ref{cor:clifford_plus},
 if $L$ is definite, we have
 $\Gamma^\times(B)=\Gamma^1(B)$. 
On the other hand, if $L$ is indefinite,
 the index $[\Gamma^\times(B) : \Gamma^1(B)]$
 is equal to $1$ or $2$, where each case can occur.
(See Section \ref{subsect:index_1_2}.)
\end{remark}

\begin{remark} \label{rem:SO_normal}
We can describe $SO(L)$ and $SO^+(L)$ 
 in a similar manner to Corollary \ref{cor:clifford_plus}, as follows
 (see Proposition \ref{prop:funct_cliff}).
Define
\begin{equation}
 \NORM(B) := \{ \alpha \in B\OverQ^\times \bigm|
 \alpha B \alpha^{-1}=B \}, \quad
 \NORM^+(B) := \{ \alpha \in \NORM(B) \bigm| \OP{Nr}(\alpha)>0 \}.
\end{equation}
We have a homomorphism $\NORM(B) \rightarrow SO(L)$
 defined by
 $\alpha \mapsto g_\alpha := (v \mapsto \alpha v \alpha^{-1})$,
 which induces isomorphisms
\begin{equation}
 SO(L) 
 \cong \NORM(B)/\Q^\times, \quad
 SO^+(L) 
 \cong \NORM^+(B) / \Q^\times.
\end{equation}
Under this correspondence,
 the spinor norm 
 of $g_\alpha \in SO(L)$
 is given as $\OP{Nr}(\alpha)$ modulo $(\Q^\times)^2$.
(See Section \ref{sect:def_cliff}
 below for spinor norm.)
\end{remark}

\subsection{Clifford algebras $\CLF{L}$ and $\CLFP{L}$}
\label{sect:def_cliff}

In this section we review Clifford algebras,
 which play an important role in our computations.
See e.g.\ \cite{O} for details of Clifford algebras.

Let $L$ be an even lattice of rank $n$.
The Clifford algebra $\CLF{L}$ of $L$ is defined as
 $\CLF{L}:=T(L)/I_L$, where
 $T(L):=\bigoplus_{k=0}^{\infty} L^{\otimes k}$,
 $L^{\otimes 0}:=\Z$
 and $I_L$ is the two-sided ideal of the tensor algebra $T(L)$
 generated by the elements of the form
 $v \otimes v-\frac{1}{2} \langle v,v \rangle$ for $v \in L$.
In particular, we have $vw+wv=\langle v,w \rangle$
 in $\CLF{L}$ for $v,w \in L$.
In order to study $\CLF{L}$,
 it is sometimes convenient to consider
 $L(1/2)$ rather than $L$.
We define
\begin{equation}
 L_0 := L(1/2), \quad
 \ANGz{~,~} := \frac{1}{2} \ANG{~,~}.
\end{equation}
Thus we have
\begin{equation}
 \OP{disc}(L_0) = \frac{1}{2^n}\OP{disc}(L),
 \quad 
 v^2=\ANGz{v,v} \quad (v \in L).
\end{equation}
The even Clifford algebra $\CLFP{L}$
 is defined as the image of the subalgebra
 $\bigoplus_{k=0}^{\infty} L^{\otimes 2k} \subset T(L)$
 in $\CLF{L}$.
Similarly, the odd part $\CLFM{L}$ of $\CLF{L}$
 is defined as the image of
 $\bigoplus_{k=0}^{\infty} L^{\otimes(2k+1)}$.
We have the decomposition
\begin{equation}
 \CLF{L} = \CLFP{L} \oplus \CLFM{L}
\end{equation}
 as a free $\Z$-module.
From the construction,
 $L$ is naturally contained in $\CLFM{L}$.
As free $\Z$-modules,
 the ranks of
 $\CLF{L}$, $\CLFP{L}$ and $\CLFM{L}$ are $2^n$, $2^{n-1}$ and $2^{n-1}$,
 respectively.
In what follows, objects on $\CLF{L}$
 are linearly extended to those on
 $\CLF{L}\OverQ$ (if possible)
 and denoted by the same symbols.

Any isometry $g \in O(L)$
 extends naturally to an automorphism of
 $\CLF{L}$ as a $\Z$-algebra.
For $\alpha \in \CLF{L}\OverQ^\times$
 satisfying $\alpha L \alpha^{-1}=L$,
 we define $g_\alpha \colon L \rightarrow L$ by
 $g_\alpha(v)=\alpha v \alpha^{-1}$.
Then $g_\alpha$ is an isometry of $L$
 because we have $g_\alpha(v)^2=v^2$.
Set
\begin{equation}
 R := \{ r \in L\OverQ \bigm|
 r^2 \neq 0 \}
 \subseteq \CLFM{L}\OverQ \cap \CLF{L}\OverQ^\times.
\end{equation}
For $r \in R$,
 the reflection of $L\OverQ$ with respect to $r$
 is written as $-g_r$ (defined over $\Q$).
Since any element in $O(L\OverQ)$
 is represented as a composition of reflections,
 we have an isomorphism
\begin{equation} \label{eq:refl_over_Q}
 \ANG{R} / \Q^\times
 = \{ \alpha \in \CLFPM{L}\OverQ^\times \bigm|
 \alpha \cdot L\OverQ \cdot \alpha^{-1}=L\OverQ \} / \Q^\times
 \CONGa 
 O(L\OverQ)
\end{equation}
 induced by $\alpha \mapsto h_\alpha$, where
 $\CLFPM{L}\OverQ^\times :=%
 (\CLFP{L}\OverQ \cup \CLFM{L}\OverQ) \cap \CLF{L}\OverQ^\times$
 and
\begin{equation}
 h_\alpha := \varepsilon_\alpha \cdot g_\alpha =
 (v \mapsto \varepsilon_\alpha \cdot \alpha v \alpha^{-1}),
 \qquad
 \varepsilon_\alpha {} :=
 \begin{cases}
   1 & \text{if} \quad \alpha \in \CLFP{L}\OverQ, \\
  -1 & \text{if} \quad \alpha \in \CLFM{L}\OverQ.
 \end{cases}
\end{equation}
We have $\det(h_\alpha)=\varepsilon_\alpha$.

\begin{definition} \label{def:spinor_map}
The involution
 $v_1 \otimes \cdots \otimes v_k \mapsto v_k \otimes \cdots \otimes v_1$
 ($v_i \in L$)
 on $T(L)$ induces
 the involution $\alpha \mapsto \CLFinvo{\alpha}$ on $\CLF{L}$,
 which is an anti-automorphism as a $\Z$-algebra.
In the same notations as above,
 the spinor norm map
 (precisely for $(L_0)_\Q$, rather than for $L_\Q$) is defined by
\begin{equation}
 \theta \colon O(L_\Q) \rightarrow \Q^\times/(\Q^\times)^2;
 \quad
 h_\alpha \mapsto
 N \alpha \, \bmod{(\Q^\times)^2},
\end{equation}
 where
 $N \alpha := \alpha \CLFinvo{\alpha} = \CLFinvo{\alpha} \alpha \in \Q^\times$
 for $\alpha$ as in (\ref{eq:refl_over_Q}).
\end{definition}

\begin{proposition} \label{prop:funct_cliff}
In the same notations as above, the map
\begin{equation}
 \{ \alpha \in \CLFPM{L}\OverQ^\times \bigm|
 \alpha L \alpha^{-1}=L \} / \Q^\times
 \rightarrow O(L);
 \quad
 \alpha \bmod{\Q^\times} \mapsto h_\alpha
\end{equation}
 is an isomorphism.
\end{proposition}

\subsection{Lattices of rank $3$}

Now we suppose that $L$ is of rank $3$. 
We fix a basis $(E_1,E_2,E_3)$ of $L$ and define
\begin{align}
 Q_L & {} := (\langle E_i,E_j \rangle)_{1 \leq i,j \leq 3}=
   \begin{pmatrix}
    2a & u & t \\
    u & 2b & s \\
    t & s & 2c
   \end{pmatrix}, \\
 \DD & {} := \disc(L)=\det(Q_L)
 =2(4abc+stu-a s^2-b t^2-c u^2),
\end{align}
 where $a,b,c,s,t,u \in \Z$.
As we wrote above, we consider $L_0 := L(1/2)$ rather than $L$:
\begin{align}
 Q_{L_0} := {} & \frac{1}{2} \, Q_L, \\
 \DDz := {} & \disc(L_0) = \det(Q_{L_0})
 =\frac{1}{8} \, \DD. 
\end{align}
We use the following basis $(e_i)$ of $\CLFP{L}$
 as a free $\Z$-module:
\begin{equation}
 B := \CLFP{L}
 = \bigoplus_{i=0}^{3} \Z e_i,
\end{equation}
 where
\begin{equation}
 e_0:=1, ~ e_1:=E_2 E_3, ~ e_2:=E_3 E_1, ~ e_3:=E_1 E_2. 
\end{equation}
We define an embedding of $B$ (as a $\Z$-algebra)
 into the matrix algebra $M_4(\Z)$ of size $4$ by
\begin{equation}
 \Phi \colon B \rightarrow M_4(\Z); \quad
 x \mapsto \Phi(x)
 \quad \text{with} \quad
 (x e_0,x e_1,x e_2,x e_3)
 =(e_0,e_1,e_2,e_3) \, \Phi(x).
\end{equation}
A direct computation shows that
 $M_i := \Phi(e_i)$ ($0 \leq i \leq 3$) are given as $M_0=1_4$,
\begin{align}
 M_1 &=
 \begin{pmatrix}
  0 & -b c & c u & -s u\\
  1 & s & 0 & u\\
  0 & 0 & 0 & b\\
  0 & 0 & -c & s
 \end{pmatrix}, \label{eq:mat_e1} \\
 M_2 &=
 \begin{pmatrix}
  0 & -s t & -a c & a s\\
  0 & t & 0 & -a\\
  1 & s & t & 0\\
  0 & c & 0 & 0
 \end{pmatrix}, \label{eq:mat_e2} \\
 M_3 &=
 \begin{pmatrix}
  0 & b t & -t u & -a b\\
  0 & 0 & a & 0\\
  0 & -b & u & 0\\
  1 & 0 & t & u
 \end{pmatrix}. \label{eq:mat_e3}
\end{align}
Hence $B$ is isomorphic
 to the $\Z$-span of the $M_i$ ($0 \leq i \leq 3$)
 as a $\Z$-algebra.

\subsection{$B := \CLFP{L}$ as a quaternion algebra}
\label{sect:def_quaternion}

In this section we study the structure of $B$ (and $B\OverQ$)
 as a quaternion algebra
 (see e.g.\ \cite{vigneras80} for details of quaternion algebras).
As a property of quaternion algebras, we have an isomorphism
 $B_\C \cong M_2(\C)$ (see the proof of Lemma \ref{lem:action}).
For $x \in B$, the trace and determinant of $x$
 as an element in $M_2(\C)$ under this isomorphism
 are called the reduced trace and norm of $x$,
 respectively. 
Similarly,
 the conjugate $\CLFinvo{x}$ is the adjugate matrix of $x$.
The fact is that the conjugate coincides with
 (the restriction of) the involution on $\CLF{L}$
 in Definition \ref{def:spinor_map}.
Hence the reduced trace and norm are $\Z$-valued on $B$.
\begin{align}
\text{Reduced trace} ~ \OP{Tr}(~) &:
 \quad \OP{Tr}(x)=x+\CLFinvo{x}=\frac{1}{2} \OP{trace}(\Phi(x)) \\
\text{Reduced norm} ~ \OP{Nr}(~) &:
 \quad \OP{Nr}(x) = N x = x \CLFinvo{x}, 
 \quad
 \OP{Nr}(x)^2=\det(\Phi(x)) \\
\text{Conjugate} &:
 \quad x \mapsto \CLFinvo{x} = \OP{Tr}(x)-x
\end{align}
We define a symmetric bilinear form on $B$ by
\begin{equation}
 \langle ~,~ \rangle_B \colon B \times B \rightarrow \Q; \quad
 (x,y) \mapsto \langle x,y \rangle_B := \frac{1}{2} \OP{Tr}(x\CLFinvo{y}).
\end{equation}
One can directly verify the following:
\begin{equation}
 Q_B:=(\langle e_i,e_j \rangle)_{0 \leq i,j \leq 3}
 =\frac{1}{2}
 \begin{pmatrix}
  2 & s & t & u\\
  s & 2 b c & s t-c u  & s u-b t\\
  t & s t-c u & 2 a c & t u-a s\\
  u & s u-b t & t u-a s & 2 a b
 \end{pmatrix}, \quad
 \det(Q_B)=\DDz^2.
\end{equation}
Let $\trzero$ denote the image of $B$
 by the orthogonal projection from $B\OverQ$ onto
\begin{equation}
 \trzeroQ := \{ x \in B\OverQ \bigm| \OP{Tr}(x)=0 \} = e_0^\bot,
\end{equation}
 that is,
\begin{equation}
 \trzero :=
 \{ x-\dfrac{1}{2}\OP{Tr}(x) \in \trzeroQ \bigm|
  x \in B \}
 =\bigoplus_{i=1}^{3} \Z \tilde{e}_i, \quad
 \tilde{e}_i:=e_i-\frac{1}{2}\OP{Tr}(e_i).
\end{equation}
By (\ref{eq:mat_e1})--(\ref{eq:mat_e3}), we have
\begin{equation}
 \tilde{e}_1=e_1-\frac{s}{2}, \quad
 \tilde{e}_2=e_2-\frac{t}{2}, \quad
 \tilde{e}_3=e_3-\frac{u}{2}.
\end{equation}
We consider $\trzero$ with the restriction of
 $\langle ~,~ \rangle_B$ as a $\Q$-valued lattice.

\begin{lemma} \label{lem:v_dot_E}
Define
\begin{equation}
 E := \frac{1}{6} \sum_{\sigma \in S_3} \OP{sgn}(\sigma) \cdot
 E_{\sigma(1)} E_{\sigma(2)} E_{\sigma(3)} \in \CLFM{L}\OverQ,
\end{equation}
 where $S_3$ is the symmetric group of degree $3$.
Then
 $Ex=xE$ for any $x \in \CLF{L}\OverQ$,
 $E^* = -E$ and $E^2=-\DDz$.
The element $E$ is independent of the choice of a $\Z$-basis $(E_i)$ of $L$,
 up to sign.
Moreover, let $(\dE_i)$ be the dual basis of $(E_i)$
 (that is, $(\dE_i)$ is a basis of $L_0^\vee$
 with $\ANGz{E_i,\dE_j}=\delta_{ij}$).
Then we have
\begin{itemize}
\item[(1)]
 $\ANG{v E,v' E}_B = \DDz \cdot \ANGz{v,v'}$
 \ for \ $v,v' \in L_0$; \\[-1.5ex]
\item[(2)]
 $\dE_i E = \tilde{e}_i$ \ for \ $1 \leq i \leq 3$; \\[-1.5ex]
\item[(3)]
 $L_0^\vee \cdot E 
 =\trzero \cong L_0^\vee(\DDz)$, \ 
 $Q_{\trzero} := (\langle \tilde{e}_i,\tilde{e}_j \rangle_B)_{1 \leq i,j \leq 3}
 = \DDz \cdot Q_{L_0}^{-1}$; \\[-1.5ex]
\item[(4)]
 $(E_1 E,E_2 E,E_3 E) %
 = (\tilde{e}_1,\tilde{e}_2,\tilde{e}_3) \, Q_{L_0}$; \\[-1.5ex] 
\item[(5)]
 $L_0 \cdot E = \DDz \cdot (\trzero)^\vee \cong L_0(\DDz)$.
\end{itemize}
Here we define
 $L_0 \cdot E := \{ v E \bigm| v \in L_0 \}$, and so on.
The map $L_0^\vee \rightarrow \trzero$
 defined by $v \mapsto v E$
 gives an isometry $L_0^\vee(\DDz) \cong \trzero$.
(This map has the ambiguity of sign, depending on $E$.)
\end{lemma}
\begin{proof}
This is a special case of Proposition \ref{prop:alternating_sum}.
For example, $\hat{E}_1$ in Proposition \ref{prop:alternating_sum}
 is equal to
 $\frac{1}{2}(E_2 E_3-E_3 E_2)=%
 E_2 E_3-\frac{1}{2} \ANG{E_2,E_3}=\tilde{e}_1$.
\end{proof}

\begin{remark}
For the element $E$ in Lemma \ref{lem:v_dot_E},
 one can directly verify
\begin{equation}
 E=E_1 E_2 E_3 + \frac{1}{2}(-s E_1+t E_2-u E_3).
\end{equation}
\end{remark}

\begin{lemma} \label{lem:psi_image}
The linear map
 $\psi \colon L_0 \cdot E \rightarrow \wedge^2 \trzero$
 defined by 
\begin{equation} \label{eq:psi_mapping}
 (\psi(E_1 E),\psi(E_2 E),\psi(E_3 E)) =
 (\tilde{e}_2 \wedge \tilde{e}_3,
  \tilde{e}_3 \wedge \tilde{e}_1,
  \tilde{e}_1 \wedge \tilde{e}_2)
\end{equation}
 is an isomorphism of $\Z$-modules,
 which is independent of the choice of a $\Z$-basis $(E_i)$ of $L$.
\end{lemma}
\begin{proof}
%
Define a bilinear form
\begin{equation}
 \trzeroQ \times \wedge^2 \trzeroQ \rightarrow \Q; \quad
 (x,w) \mapsto (x \wedge w) \cdot \Omega, \quad
 \text{where} \quad
 \Omega := \frac{\DDz}{\tilde{e}_1 \wedge \tilde{e}_2 \wedge \tilde{e}_3}.
\end{equation}
Here $\Omega$ is independent of the basis $(E_i)$,
 that is, if we transform $(E_i)$ by $\varepsilon \in GL_3(\Q)$,
 then both the numerator and denominator of $\Omega$
 are multiplied by $\det(\varepsilon)^2$.
This bilinear form induces a map
\begin{equation}
 \psi \colon
 L_0 \cdot E \rightarrow {\OP{Hom}_\Q(\trzeroQ,\Q)}
 \CONGa {\wedge^2 \trzeroQ}; \quad
 v E \mapsto \ANG{-,v E}_B \mapsto w, \label{eq:def_psi1}
\end{equation}
 where $\ANG{-,v E}_B = (-\wedge w) \cdot \Omega$, that is,
\begin{equation}
 \ANG{y,v E}_B = (y \wedge w) \cdot \Omega
 \quad (\forall y \in \trzeroQ). \label{eq:def_psi2}
\end{equation}
When $v=E_i$ and $y=\dE_j E=\tilde{e}_j$, we should have
\begin{equation*}
 ( \tilde{e}_j \wedge w ) \cdot \Omega
 = \ANG{\dE_j E,E_i E}_B=\DDz \cdot \delta_{ij}
\end{equation*}
 by Lemma \ref{lem:v_dot_E}(1).
For example, if $i=1$,
 this implies
 $w=\tilde{e}_2 \wedge \tilde{e}_3$,
 i.e.\ $\psi(E_1 E)=\tilde{e}_2 \wedge \tilde{e}_3$.
Other cases are similar and the map $\psi$ satisfies (\ref{eq:psi_mapping}).
By the construction, $\psi$ does not depend on $(E_i)$.
\end{proof}

\begin{remark}
In Lemma \ref{lem:psi_image},
 if we equip the second exterior power $\wedge^2 \trzero$
 of the $\Q$-valued lattice $\trzero$
 with a natural bilinear form
 (by the theory of exterior product),
 the map $\psi$ is an isometry.
\end{remark}

\subsection{Lattice $W := \wedge^2 B \cong U^{\OSUM 3}$}

We consider $\wedge^2 B \cong \Z^6$ as a lattice by the following:
\begin{equation}
 W := \wedge^2 B, \quad
 \langle w_1,w_2 \rangle_W :=
 \frac{w_1 \wedge w_2}{\omega} \quad
 (w_1,w_2 \in W), \quad
 \omega := e_0 \wedge e_1 \wedge e_2 \wedge e_3.
\end{equation}
In this definition we do not use the structure of $B$
 as a quaternion algebra.
Similarly to $\Omega$ in the previous section,
 $\omega$ is independent of the choice of the basis $(E_i)$.
In fact, if we transform $(E_i)$ by $\varepsilon \in GL_3(\Z)$,
 then $\omega$ is multiplied by $\det(\varepsilon)^2=1$.
Set
\begin{equation}
 w=\sum p_{ij} \cdot e_i \wedge e_j \in W
 \quad \text{with} \quad
 (p_{01},p_{02},p_{03},p_{23},p_{31},p_{12}) \in \Z^6.
\end{equation}
Then we have
\begin{equation}
 \langle w,w \rangle_W=
 2( p_{01} p_{23}+p_{02} p_{31}+p_{03} p_{12} ).
\end{equation}
Thus $W \cong U^{\OSUM 3}$.
We define sublattices $P^+$ and $P^-$ of $W$ as follows:

\begin{lemma} \label{lem:def_Psi}
Define
\begin{equation}
 \Psi^{\pm} \colon L_0 \cdot E \rightarrow W\OverQ; \quad
 v E \mapsto e_0 \wedge v E \pm \psi(v E).
\end{equation} 
(For the definition of $\psi$,
 see Lemma \ref{lem:psi_image} and its proof.)
Set $P^{\pm} := \Psi^{\pm}(L_0 \cdot E)$.
Then the following hold:
\begin{itemize}
\item[(1)]
$P^+$ and $P^-$ are primitive sublattices of $W$; \\[-1.5ex]
\item[(2)]
$(P^\pm)^\bot=P^{\mp}$ in $W$; \\[-1.5ex]
\item[(3)]
$P^{\pm} \cong L_0(\pm 2) = L(\pm 1)$; \\[-1.5ex]
\item[(4)]
the map $L_0(\pm 2) \rightarrow P^{\pm}$ defined by
 $v \mapsto \Psi^{\pm}(v E)$
 is an isometry.
\end{itemize}
\end{lemma}
\begin{proof}
By Lemma \ref{lem:psi_image}, we have
\begin{equation}
 w_i^{\pm} := \Psi^{\pm}(E_i E)
 =e_0 \wedge E_i E \pm \tilde{e}_j \wedge \tilde{e}_k,
\end{equation}
 where $(i,j,k)=(1,2,3),(2,3,1),(3,1,2)$.
For example, 
 $w_1^+$ is computed as
\begin{equation*}
 w_1^+ = e_0 \wedge (a e_1+\dfrac{u}{2} e_2+\dfrac{t}{2} e_3)
    +(e_2-\dfrac{t}{2} e_0) \wedge (e_3-\dfrac{u}{2} e_0)
 = a e_{01}+u e_{02}+e_{23}.
\end{equation*}
(Recall that $\tilde{e}_i$ is defined as
 $e_i - \frac{1}{2} \OP{Tr}(e_i) \cdot e_0$.)
Other cases are similar and we obtain
\begin{equation} \label{eq:mat_w_i}
 \begin{pmatrix}
  w_1^+ \\ w_2^+ \\ w_3^+ \\ w^{-}_1 \\ w^{-}_2 \\ w^{-}_3
 \end{pmatrix}
 =
 \begin{pmatrix}
  a & u & 0 & 1 & 0 & 0\\
  0 & b & s & 0 & 1 & 0\\
  t & 0 & c & 0 & 0 & 1\\
  a & 0 & t & -1 & 0 & 0\\
  u & b & 0 & 0 & -1 & 0\\
  0 & s & c & 0 & 0 & -1
 \end{pmatrix}
 \begin{pmatrix}
  e_{01} \\ e_{02} \\ e_{03} \\ e_{23} \\ e_{31} \\ e_{12}
 \end{pmatrix},
 \quad \text{where} \quad
 e_{ij} := e_i \wedge e_j.
\end{equation}
In particular, we have $w_i^\pm \in W$.
From the shape of the matrix in (\ref{eq:mat_w_i}),
 the primitivity of $P^{\pm}$ in $W$ follows.
Thus the assertion (1) is proved.
For $v,v' \in L_0$, we have
\begin{align*}
 e_0 \wedge v E \wedge \psi(v' E)
 &= e_0 \wedge \frac{\ANG{v E,v' E}_B}{\Omega} \\
 &= \frac{\ANG{v E,v' E}_B}{\DDz} \cdot
  e_0 \wedge \tilde{e}_1 \wedge \tilde{e}_2 \wedge \tilde{e}_3 \\
 &= \ANGz{v,v'} \cdot \omega.
\end{align*}
(Recall that the map $\psi$ is defined by (\ref{eq:def_psi2}).)
This implies that $P^+$ and $P^-$ are orthogonal to each other
 because we have
\begin{equation*}
  \Psi^+(v E) \wedge \Psi^-(v' E)
  = e_0 \wedge v' E \wedge \psi(v E)
   -e_0 \wedge v E  \wedge \psi(v' E) \\
 =0.
\end{equation*}
Hence the assertion (2) follows.
Similarly the assertions (3) and (4) are verified.
(Alternatively, one can directly check the assertions (2)--(4) from
 (\ref{eq:mat_w_i}).)
\end{proof}

\subsection{Action of $B$ on $W$}
\label{subsect:hamilton}

Let $\mu$ denote the  two-sided action of $B$ on $W$ defined by
\begin{equation}
 \mu(x,y) \colon W \rightarrow W; \quad
 h_1 \wedge h_2 \mapsto x h_1 y \wedge x h_2 y,
\end{equation}
 where $x,y,h_1,h_2 \in B$.
Since $\langle ~,~ \rangle_W$ is defined by an identification
 $\wedge^4 B=\Z$,
 we have
\begin{equation}
 \langle \mu(x,y) \cdot w_1,\mu(x,y) \cdot w_2 \rangle_W
 =\OP{Nr}(x)^2 \cdot \OP{Nr}(y)^2 \cdot \langle w_1,w_2 \rangle_W
 \quad \text{for} \quad
 w_1, w_2 \in W.
\end{equation}
(For the property of $\OP{Nr}(~)$,
 see Section \ref{sect:def_quaternion}.)


\begin{lemma} \label{lem:action}
$P^+$ and $P^-$ are (respectively) stable under the action $\mu$ of $B$.
Moreover, we have
\begin{equation} \label{eq:left-right-lem}
 \mu(x,1) |_{P^+} = \OP{Nr}(x) \cdot \OP{id}_{P^+}, \quad
 \mu(1,x) |_{P^-} = \OP{Nr}(x) \cdot \OP{id}_{P^-}
 \quad \text{for} \quad x \in B.
\end{equation}
\end{lemma}
\begin{proof}
We verify the statement over $\C$.
We may assume that
\begin{equation*}
 L_0=(-1)^{\OSUM 3},
 \quad \text{that is,} \quad
 a=b=c=-1, ~ s=t=u=0,
\end{equation*}
 because the change of the basis $(E_i)$ over $\C$
 does not change the $\C$-linear extension of the map $\psi$
 (see the proof of Lemma \ref{lem:psi_image}).
Define
\begin{equation*}
 H := B_\C
 = \C \oplus \C i \oplus \C j \oplus \C k, \quad
 i := e_1, ~ j:= e_2, ~ k := e_3,
\end{equation*}
 where $i^2=j^2=-1$ and $k=ij=-ji$.
By (\ref{eq:mat_w_i}),
 the basis $(w_i^{\pm})$ of $P^{\pm}$ is given by
\begin{equation*}
 w^{\pm}_1 = -1 \wedge i \pm j \wedge k, ~
 w^{\pm}_2 = -1 \wedge j \pm k \wedge i, ~
 w^{\pm}_3 = -1 \wedge k \pm i \wedge j.
\end{equation*}
(For example, one can directly check
\begin{equation*}
 \mu(1,x) \cdot (1 \wedge i+j \wedge k)
 =x \wedge ix+jx \wedge kx
 =1 \wedge i+j \wedge k
 \quad \text{for} \quad  x \in \{ i,j,k \}.
\end{equation*}
Hence the statement of the lemma holds for this case.)

We fix an identification $H=M_2(\C)$
 by the correspondence
\begin{equation*}
 1 \leftrightarrow \SqMat{1}{0}{0}{1}, ~
 i \leftrightarrow \SqMat{\sqrt{-1}}{0}{0}{-\sqrt{-1}}, ~
 j \leftrightarrow \SqMat{0}{1}{-1}{0}, ~
 k \leftrightarrow \SqMat{0}{\sqrt{-1}}{\sqrt{-1}}{0}.
\end{equation*}
Then, for $x=\SqMat{\alpha}{\beta}{\gamma}{\delta} \in M_2(\C)$,
 we have
\begin{equation*}
 \OP{Tr}(x)=\OP{trace}(x)=\alpha+\delta, \quad
 \OP{Nr}(x)=\det(x)=\alpha \delta-\beta \gamma, \quad
 \CLFinvo{x}=\SqMat{\delta}{-\beta}{-\gamma}{\alpha}.
\end{equation*}
Under the basis
 $(u_\nu)=\left(%
 u_1 := (\SqMatSmall{1}{0}{0}{0}), %
 u_2 := (\SqMatSmall{0}{1}{0}{0}), %
 u_3 := (\SqMatSmall{0}{0}{1}{0}), %
 u_4 := (\SqMatSmall{0}{0}{0}{1}) \right)$
 of $M_2(\C)$, we can write $w_1^+$ as
\begin{align*}
 w_1^+ &=  -1 \wedge i+j \wedge k \notag \\
  &=-(u_1+u_4) \wedge \sqrt{-1}(u_1-u_4)
   +(u_2-u_3) \wedge \sqrt{-1}(u_2+u_3) \\
  &=2 \sqrt{-1}(u_1 \wedge u_4+u_2 \wedge u_3). \notag
\end{align*}
Similarly, we have
 $w^+_2 = 2(u_1 \wedge u_3+u_2 \wedge u_4)$ and
 $w^+_3 = 2\sqrt{-1} (-u_1 \wedge u_3+u_2 \wedge u_4)$.
Hence
\begin{align*}
 (P^+)_\C &= \{
  u_1 \wedge u_3, ~
  u_2 \wedge u_4, ~
  u_1 \wedge u_4+u_2 \wedge u_3
 \}_\C, \\
 (P^-)_\C &= \{
  u_1 \wedge u_2, ~
  u_3 \wedge u_4, ~
  u_1 \wedge u_4-u_2 \wedge u_3
 \}_\C,
\end{align*}
 where $\{ ~ \}_\C$ means the span over $\C$ and
 the second equality follows from $P^-=(P^+)^\bot$.
This result implies (\ref{eq:left-right-lem})
 by a direct computation.
In fact, we have
\begin{align*}
 & \mu(x,1) \cdot (u_1 \wedge u_4+u_2 \wedge u_3) \\
 &= (\alpha u_1+\gamma u_3) \wedge (\beta u_2+\delta u_4)
 +(\alpha u_2+\gamma u_4) \wedge (\beta u_1+\delta u_3) \\
 &= (\alpha \delta-\beta \gamma) (u_1 \wedge u_4+u_2 \wedge u_3).
\end{align*}
Other cases are similar and omitted.
Therefore $P^+=(P^-)^\bot$ is stable under $\mu(1,y)$
 for $y \in B$ with $\OP{Nr}(y) \neq 0$.
In the case where $\OP{Nr}(y)=0$, we chose $\lambda \in \C$ such that
 $y_1 :=y - \lambda \cdot 1$ satisfies $\OP{Nr}(y_1) \neq 0$.
Then $\mu(1,y_1) \cdot P^+ \subseteq P^+$, thus
 $\mu(1,y) \cdot P^+ \subseteq P^+$.
Hence it follows that $P^+$ is stable under the action $\mu$.
One can apply a similar argument to $P^-$,
 and the lemma is proved.
\end{proof}

\subsection{Duality of $\CLFP{L}$ and $\CLFM{L}$} \label{subsect:dual}

Under the following pairing,
 $B=\CLFP{L}$ and $\CLFM{L}$ are dual to each other:
\begin{equation} \label{eq:duality_plus_minus}
 \CLFP{L} \times \CLFM{L} \rightarrow \Z; \quad
 (x,y) \mapsto (x,y)_E :=
 \frac{\tau(xy^*)}{E_1 \wedge E_2 \wedge E_3}.
\end{equation}
Here $\tau$ is the following natural projection:
\begin{equation}
 \tau \colon \CLF{L} \rightarrow
 \frac{\CLF{L}}{\text{the image of $\bigoplus_{k=0}^{2} L^{\otimes k}$}}
 \cong \wedge^3 L.
\end{equation}
The pairing $(~,~)_E$ has the ambiguity of sign,
 depending on $(E_i)$ (or $E$).
One can easily check
 the following intersection matrix of $(~,~)_E$,
 which implies the duality
 in (\ref{eq:duality_plus_minus}):
\begin{equation}
 ( \{ e_0,e_1,e_2,e_3 \} , \{ E_1 E_2 E_3,E_1,E_2,E_3 \} )_E =
 \begin{pmatrix}
 -1 & 0 & 0 & 0 \\
  * & 1 & 0 & 0 \\
  * & 0 & 1 & 0 \\
  * & 0 & 0 & 1
 \end{pmatrix}.
\end{equation}
(See also Remark \ref{rem:clf_minus_mat}.)
We can also write
 $(x,y)_E= \DDz^{-1} \cdot \ANG{x,y E}_B$.
In fact, Lemma \ref{lem:v_dot_E} implies that
\begin{equation}
 \DDz^{-1} \cdot \ANG{\tilde{e}_i,E_j E}_B
 = \delta_{ij} = (\tilde{e}_i,E_j)_E \qquad
 (E_0 := - E, ~ \tilde{e}_0 := e_0, ~ 0 \leq i,j \leq 3).
\end{equation}
Here we use $E^*=-E$.

Since $W=\wedge^2 \CLFP{L}$ is unimodular (or self-dual) as a lattice,
 the duality by the pairing $(~,~)_E$ induces an isomorphism
\begin{equation}
 \iota \colon
 \wedge^2 \CLFP{L}=W
 \cong W^\vee = \operatorname{Hom}(W,\Z)
 \cong \wedge^2 \CLFM{L} =: W',
\end{equation}
 which is independent of
 the choice of a $\Z$-basis $(E_i)_{1 \leq i \leq 3}$ of $L$.
For example,
 by the correspondence
 $\tilde{e}_0 \wedge \tilde{e}_1 \leftrightarrow %
 \tilde{e}_2 \wedge \tilde{e}_3  \leftrightarrow
 E_2 \wedge E_3$,
 we have
 $\iota(\tilde{e}_0 \wedge \tilde{e}_1)=E_2 \wedge E_3$.
We define a right action $\tilde{\mu}$ on $W$ by $\CLFM{L}^\times$:
\begin{equation}
 \tilde{\mu}(x) \colon W \rightarrow W; \quad
 h_1 \wedge h_2 \mapsto \iota^{-1}(h_1 x \wedge h_2 x),
\end{equation}
 where $x \in \CLFM{L}^\times$ and $h_1,h_2 \in \CLFP{L}$.

\begin{lemma} \label{lem:clifford_minus_action}
For $x \in \CLFM{L}^\times$, we have
\begin{equation}
 \tilde{\mu}(x) |_{P^+} = (-Nx) \cdot \eta_x, \quad
 \tilde{\mu}(x) |_{P^-} = (-Nx) \cdot \OP{id}_{P^-}.
\end{equation}
Here we identify $P^+$ and $L$ by $w_i^+=E_i$ and
 $\eta_x(v)=-x^{-1}vx$ for $v \in L$.
\end{lemma}
\begin{proof}
We can work over $\Q$ and
 may assume that $(E_i)_{1 \leq i \leq 3}$ is an orthogonal basis of $L_\Q$
 by transforming it under $SL_3(\Q)$.
Then $E_0 := -E = E_1 E_2 E_3$.
Observe the compatibility of the assertion in the sense that
 $\tilde{\mu}(x) \circ \tilde{\mu}(y) = \tilde{\mu}(yx)$
 and
 $\eta_x \circ \eta_y = \eta_{yx}$ hold.
Hence it is enough to check the assertion for $x=E_1$ 
 because any element in $\CLFM{L}^{\times}$ is written as
 the product of some elements in $L_\Q$.
The action $\tilde{\mu}(E_1)$ under the bases
 $(e_i \wedge e_j)$ and $(E_i \wedge E_j)$
 is given by the following:
\begin{equation*}
 \begin{array}{ccrcr}
  e_0 \wedge e_1 & \mapsto & E_0 \wedge E_1     & \leftrightarrow & e_2 \wedge e_3  \\
  e_0 \wedge e_2 & \mapsto & -a E_3 \wedge E_1  & \leftrightarrow & -a e_0 \wedge e_2 \\
  e_0 \wedge e_3 & \mapsto & -a E_1 \wedge E_2  & \leftrightarrow & -a e_0 \wedge e_3 \\
  e_2 \wedge e_3 & \mapsto & a^2 E_2 \wedge E_3 & \leftrightarrow & a^2 e_0 \wedge e_1 \\
  e_3 \wedge e_1 & \mapsto & -a E_0 \wedge E_2  & \leftrightarrow & -a e_3 \wedge e_1 \\
  e_1 \wedge e_2 & \mapsto & -a E_0 \wedge E_3  & \leftrightarrow & -a e_1 \wedge e_2 
 \end{array}
\end{equation*}
(For example, the first line means
 $e_0 E_1 \wedge e_1 E_1=E_0 \wedge E_1$
 and $\iota^{-1}(E_0 \wedge E_1)=e_2 \wedge e_3$.)
From this, the assertion for $x=E_1$ directly follows.
(See (\ref{eq:mat_w_i}) for $w_i^\pm$.)
\end{proof}

\begin{remark} \label{rem:clf_minus_mat}
(i)
One can also define a left action on $W$ by $\CLFM{L}^\times$
 in a similar manner.
(ii)
$\CLFM{L}^\times$ may or may not be an empty set.
(See Section \ref{subsect:appendix_exm}
 for an example of an  element in $\CLFM{L}^\times$.)
(iii)
By direct computations, it follows that
 the intersection matrix of $\CLFM{L}$
 with the quadratic form $N x=x x^*$,
 for the basis $(E_1 E_2 E_3,E_1,E_2,E_3)$,
 is given by
\begin{equation}
 \frac{1}{2}
 \begin{pmatrix}
 2 a b c & a s & s u-b t & c u \\
 a s & 2 a & u & t \\
 s u-b t & u & 2 b & s \\
 c u & t & s & 2 c
 \end{pmatrix}.
\end{equation}
If we take another basis $(-E_1 E_2 E_3-t E_2,E_1,E_2,E_3)$,
 the intersection matrix is given by
\begin{equation}
 \frac{1}{2}
\begin{pmatrix}
 2 ( s t u+a b c)  & -( t u+a s)  & -( s u+b t)  &  -( c u+s t) \\
 -( t u+a s)  & 2 a & u & t \\
 -( s u+b t)  & u & 2 b & s \\
 -( c u+s t)  & t & s & 2 c
 \end{pmatrix}
 = \DDz \cdot Q_B^{-1}.
\end{equation}
Hence the second basis of $\CLFM{L}$
 is the dual basis of the basis $(e_0,e_1,e_2,e_3)$ of $\CLFP{L}$
 with respect to $(~,~)_E$.
\end{remark}

\subsection{Proof of Theorem \ref{thm:MAIN}}

We have constructed the following objects:
\begin{itemize}
\item[$\circ$]
A lattice $W := \wedge^2 B \cong U^{\OSUM 3}$
 with a two-sided action $\mu$ by $B$ (as a $\Z$-module)
\item[$\circ$]
Primitive sublattices $P^+$, $P^- \subseteq W$;
 $(P^{\pm})^\bot=P^{\mp}$, $P^{\pm} \cong L_0(\pm 2)$
\item[$\circ$]
An element
 $E \in \CLFM{L}_\Q$ in the center of $\CLF{L}$;
 $E$ is independent of the basis $(E_i)$ of $L$, up to sign;
 $E^2=-\DDz$, $E^*=-E$ 
\item[$\circ$]
 $L_0 \cdot E := \{ v E \bigm| v \in L_0 \} \subset B_\Q$
\item[$\circ$]
 $\psi \colon L_0 \cdot E \rightarrow W\OverQ$,
 which is independent of $(E_i)$
\item[$\circ$]
$\Psi^{\pm} \colon L_0 \cdot E \rightarrow P^{\pm}$,
 $\Psi^{\pm}(v E) = e_0 \wedge v E \pm \psi(v E)$
\item[$\circ$]
An isometry
 $L_0(\pm 2) = L(\pm 1) \rightarrow P^{\pm}$,
 which is defined by
 $v \mapsto \Psi^{\pm}(v E)$
\end{itemize}

\begin{lemma}[cf.\ {\cite[\S 1]{Shioda}}] \label{lem:sl4_map}
Let $V$ 
 be a free $\Z$-module of rank $4$.
We fix an identification $\wedge^4 V=\Z$
 and consider $\sW := \wedge^2 V$ as a lattice,
 equipped with the bilinear form
 $\sW \times \sW \rightarrow \Z$ 
 defined by
 $(x,y) \mapsto x \wedge y$.
Then the natural map
\begin{equation}
 PSL(V) \cong PSL_4(\Z) \rightarrow SO^+(\sW)
\end{equation}
 is an isomorphism.
In particular,
 for $\varphi \in GL(V_\R)$ and $V' := \varphi(V)$,
 if $\wedge^2 V'=\sW$ in $\sW_\R$, then $V'=V$.
\end{lemma}
\begin{proof}
It is classically known that $SL_n(\Z)$
 is a maximal discrete subgroup of $SL_n(\R)$
 (see e.g.\ \cite{stack}).
Combined with the natural isomorphism
 $PSL(V_\R) \cong SO^+(\sW_\R)$,
 this fact (for $n=4$) implies the lemma.
 \end{proof}

\begin{proof}[Proof of Theorem {\ref{thm:MAIN}}]
We divide the proof into two cases
 according to the sign of the determinant of $g \in O(L)$.
\\

{\bf Case I:} $\det(g)=1$.
By Proposition \ref{prop:funct_cliff},
 any $g \in SO(L)$
 is of the form $g=g_\alpha$ with $\alpha \in B\OverQ^\times$
 satisfying $\alpha L \alpha^{-1}=L$,
 where $g_\alpha$ is defined as the map
 $v \mapsto \alpha v \alpha^{-1}$.
Set
\begin{equation*}
 \nu := |\OP{Nr}(\alpha)|^{1/2}, \quad
 \alpha_0 := \alpha/\nu \in B_\R, \quad
 \varepsilon := \OP{Nr}(\alpha_0) \in \{ \pm 1 \}.
\end{equation*}
We extend the action $\mu$ linearly to that over $\R$.
Then we have $\mu(\alpha,\alpha^{-1})=\mu(\alpha_0,\alpha_0^{-1})$.
By Lemma \ref{lem:def_Psi},
 the following map is an isometry:
\begin{equation*}
 \lambda^+ \colon
 L_0(2)=L \rightarrow P^+; \quad
 v \mapsto  \Psi^{+}(v E)
 = e_0 \wedge v E + \psi(v E).
\end{equation*}

{\bf Claim 1:}
 $\lambda^+ \circ g_\alpha \circ (\lambda^+)^{-1} = \mu(\alpha,\alpha^{-1})|_{P^+}$.
Proof:
By the construction of the map
 $\psi$ 
 (see the proof of Lemma \ref{lem:psi_image}), we have
\begin{equation*}
 \psi( g_\alpha(v) \cdot E)
 = \psi( \alpha v \alpha^{-1} \cdot E)
 = \psi( \alpha \cdot v E \cdot \alpha^{-1} )
 = \mu(\alpha,\alpha^{-1}) \cdot \psi(v E)
\end{equation*}
for $v \in L$.
Thus
\begin{equation*}
 \lambda^+( g_\alpha(v) )
 = e_0 \wedge ( g_\alpha(v) \cdot E )+\psi( g_\alpha(v) \cdot E ) \\
 = \mu(\alpha,\alpha^{-1}) \cdot \lambda^+(v). 
\end{equation*}
Claim 1 immediately follows from this.

{\bf Claim 2:}
 $\alpha \in B^\times ~ \Rightarrow ~ g_\alpha \in SO_\GAM(L)$.
Proof:
Assume $\alpha \in B^\times$.
Then $\varepsilon = \OP{Nr}(\alpha)$.
By Lemma \ref{lem:action}, we have
\begin{gather*}
 \mu(\alpha,\alpha^{-1})|_{P^+}
 = \varepsilon \cdot \mu(1,\alpha^{-1})|_{P^+}
 = \sigma|_{P^+}, \\ 
 \text{where} \quad
 \sigma := \varepsilon \cdot \mu(1,\alpha^{-1}) \in O(W).
\end{gather*}
On the other hand, we have
 $\sigma|_{P^-} = \id_{P^-}$,
 again by Lemma \ref{lem:action}.
Recall that $P^+$ and $P^-$ are primitive sublattices of
 the even unimodular lattice $W$
 satisfying $(P^+)^\bot=P^-$.
By the theory of discriminant forms \cite{N},
 we have a natural isomorphism $A(P^+) \cong A(P^-)$,
 which is compatible with the action of $\sigma$.
Hence $\sigma$ acts on $A(P^+)$ trivially.
Thus $g_\alpha \in SO_\GAM(L)$ by Claim~1.

{\bf Claim 3:}
 $g_\alpha \in SO_\GAM(L) ~ \Rightarrow ~ \alpha_0 \in B^\times$.
Proof:
Assume $g_\alpha \in SO_\GAM(L)$.
By Claim 1,
 the action of $\mu(\alpha,\alpha^{-1})$
 on $A(P^+)$ is trivial.
Similarly to Claim 2, we have
\begin{gather*}
 \mu(\alpha,\alpha^{-1})|_{P^+}
 =\mu(\alpha_0,\alpha_0^{-1})|_{P^+}
 =\varepsilon \cdot \mu(1,\alpha_0^{-1})|_{P^+}
 =\tau|_{P^+}, \\
 \text{where} \quad
 \tau := \varepsilon \cdot \mu(1,\alpha_0^{-1}) \in O(W_\R).
\end{gather*}
By Lemma \ref{lem:action},
 $\tau$ acts on $(P^-)_\R$ trivially.
Hence $\tau$ preserves $W$
 (by the theory of discriminant forms),
 and so does $\mu(1,\alpha_0^{-1})$.
Now Lemma \ref{lem:sl4_map} implies
 $B \cdot \alpha_0^{-1}=B$.
Thus we have $\alpha_0 \in B^\times$
 by $\alpha_0 \in B$ and $\OP{Nr}(\alpha_0)= \pm 1$.
%
\\

{\bf Case II:} $\det(g)=-1$.
The proof for this case is similar to Case I
 and we give only an outline.
Any $g \in O(L)$ with $\det(g)=-1$
 is of the form $-g_\alpha$,
 where $g_\alpha$ is as in Case I.
Hence
\begin{equation*}
 g=- g_\alpha=\eta_\beta := (v \mapsto - \beta^{-1} v \beta),
 \quad \beta := \alpha^{-1} \cdot E \in \CLFM{L}_\Q^\times.
\end{equation*}
(Recall that $E \in \CLFM{L}^\times_\Q$
 is an element in the center of $\CLF{L}$.)
We use $\tilde{\mu}$ defined in Section \ref{subsect:dual}.

(i)
Assume
\begin{equation*}
 \beta \in \CLFM{L}^\times
 = \{ \beta \in \CLFM{L} \bigm| N \beta=\pm 1 \}.
\end{equation*} 
Then we have $\beta^{-1} L \beta=L$, and hence $\eta_\beta \in O(L)$,
 because $L$ is orthogonal to $E$ in $\CLFM{L}_\Q$
 with respect to the quadratic form $N x=x x^*$.
By Lemma \ref{lem:clifford_minus_action},
 $\sigma := (-N \beta) \cdot \tilde{\mu}(\beta) \in O(W)$
 acts on $P^+ \cong L$ as $\eta_\beta$ and on $P^-$ as the identity.
Hence $\eta_\beta \in O_\GAM(L)$.

(ii)
Conversely, assume $\eta_\beta \in O_\GAM(L)$.
Set
\begin{equation*}
 \beta_0 := \beta/|N \beta|^{1/2}, \quad
 \tau := (-N \beta_0) \cdot \tilde{\mu}(\beta_0) \in O(W_\R).
\end{equation*}
Here we extend $\tilde{\mu}$ linearly over $\R$.
By Lemma \ref{lem:clifford_minus_action},
 $\tau$ acts on $P^+$ as $\eta_{\beta_{0}}=\eta_\beta$
 and on $P^-$ as the identity.
Thus $\tau(W)=W$.
By Lemma \ref{lem:sl4_map},
 we have $\CLFP{L} \cdot \beta_0 = \CLFM{L}$,
 and hence $\beta_0 \in \CLFM{L}^\times$.
\end{proof}

\begin{remark}
This proof is motivated by the geometry of
 a certain kind of $2$-dimensional complex tori, 
 which is sometimes called QM (quaternion multiplication) type.
\end{remark}

\section{Application to $U(k) {\OSUM} \LL{2l}$} \label{sect:LAT_kl}

In this section,
 we apply the results
 in the previous section 
 to study the orthogonal group of the following lattice $\LAT$:
\begin{equation}
 \LAT :=
 \begin{pmatrix}
  0 & 0 & k \\
  0 & 2l & 0 \\
  k & 0 & 0
 \end{pmatrix}
 \cong U(k) {\OSUM} \LL{2l}.
\end{equation}
Here $k$ and $l$ are non-zero integers.
We identify the orthogonal groups of $\LAT$ and $\LAT_0 := \LAT(1/2)$
 in a natural way.

\subsection{Even Clifford algebra $\OP{Cl}^+(\LAT)$} \label{subsect:U_kl}

In our setting,
 $\OP{Cl}^+(\LAT)$
 is isomorphic to the $\Z$-subalgebra of
 $M_4(\Z)$ spanned by
\begin{equation*}
 M_0=I_4, ~
 M_1=
 \begin{pmatrix}
  * & * & 0 & 0\\
  * & * & 0 & 0\\
  * & * & 0 & l\\
  * & * & 0 & 0
 \end{pmatrix}, ~
 M_2=
 \begin{pmatrix}
  * & * & 0 & 0\\
  * & * & 0 & 0\\
  * & * & k & 0\\
  * & * & 0 & 0
 \end{pmatrix}, ~
 M_3=
 \begin{pmatrix}
  * & * & 0 & 0\\
  * & * & 0 & 0\\
  * & * & 0 & 0\\
  * & * & k & 0
 \end{pmatrix},
\end{equation*}
 where each $M_i$ corresponds to $e_i$
 (see (\ref{eq:mat_e1})--(\ref{eq:mat_e3})).
Hence
\begin{equation} \label{eq:clf_LAT_B}
 \CLFP{\LAT} \cong \clfC :=
 \{ \alpha = \SqMat{a}{b}{c}{d} \in M_2(\Z) \bigm| 
 a-d \equiv c \equiv 0 \bmod{k}, ~
 b \equiv 0 \bmod{l}
 \}.
\end{equation}
In what follows,
 we identify $\OP{Cl}^+(\LAT)$ with $\clfC$ by
\begin{equation}
 e_1 = \SqMat{0}{l}{0}{0}, \quad
 e_2 = \SqMat{k}{0}{0}{0}, \quad
 e_3 = \SqMat{0}{0}{k}{0}.
\end{equation}
Then we have
 $\OP{Nr}(\alpha)=\det(\alpha)$, and hence
 $\clfC^\times = \clfC \cap GL_2(\Z)$
 and $\clfC^1 = \clfC \cap SL_2(\Z)$.
In this setting,
 Theorem \ref{thm:MAIN} implies the following:
\begin{proposition} \label{prop:even_clifford_LAT}
We have an isomorphism
\begin{equation}
 \PGB{\clfC} := \clfC^\times / \{ \pm 1 \}
 \CONGa
 SO_\GAM(\LAT); \quad
 \alpha \mapsto (\sigma \mapsto \alpha \sigma \alpha^{-1})
\end{equation}
Moreover, $\PSB{\clfC} := \clfC^1 / \{ \pm 1 \}$
 corresponds to $SO^+_\GAM(\LAT)$ under this isomorphism.
\end{proposition}

\begin{lemma} \label{lem:condition_negative_det}
We have $\clfC^\times \subseteq SL_2(\Z)$
 if and only if $-1$ is not a quadratic residue modulo $k$.
(This condition is also equivalent to the condition that
 $k$ is divisible by $4$ or a prime $p$ such that
 $p \equiv 3 \bmod 4$.)
\end{lemma}
\begin{proof}
It is enough to show
 the following group homomorphism is surjective:
\begin{equation*}
\begin{split}
 \clfC^\times
  = \{ \alpha = \SqMat{a}{b}{c}{d} \in \clfC \bigm| {} & {}
  \OP{Nr}(\alpha)=\pm 1 \} \\
  \rightarrow {} & {} \{ \bar{a} \in (\Z/k \Z)^\times \bigm|
   a^2 \equiv \pm 1 \bmod k \};
  \quad \alpha \mapsto \bar{a}.
\end{split}
\end{equation*}
This is the case because the projection
 $(\Z/kl \Z)^\times \rightarrow (\Z/k \Z)^\times$ is surjective,
 and so is the following group homomorphism:
\begin{equation*}
\begin{split}
 \{ \alpha & {} = \begin{pmatrix} a& b \\ c& d \end{pmatrix}
  \in GL_2(\Z) \bigm|
  c \equiv 0 \bmod k, ~ b\equiv 0 \bmod l \} \\
 & \quad \rightarrow
  \{ (\bar{a},\bar{d}) \bigm|
  \bar{a},\bar{d} \in (\Z/kl \Z)^\times, ~
  ad \equiv \pm 1 \bmod kl \}; \quad
  \alpha \mapsto (\bar{a},\bar{d}).
\end{split}
\end{equation*}
\end{proof}

By Lemma \ref{lem:v_dot_E},
 there is a natural identification (up to sign):
\begin{equation} \label{eq:B_and_L_dotE}
 \LAT_0(\DDz) = \LAT_0 \cdot E; \quad
 v \leftrightarrow v E,
\end{equation}
 where $\DDz := \disc(\LAT_0) = - \frac{1}{4} k^2 l$.
Again, by Lemma \ref{lem:v_dot_E}, we have
\begin{equation} \label{eq:Ei_Lat_kl}
 (E_i E)_{1 \leq i \leq 3}
 = (\tilde{e}_1,\tilde{e}_2,\tilde{e}_3) \, Q_{\LAT_0}
 = ((k/2) \tilde{e}_3,l \tilde{e}_2,(k/2) \tilde{e}_1).
\end{equation}
(Recall that $\tilde{e}_i = e_i-\frac{1}{2} \tr(e_i)$.)
Thus
\begin{equation} \label{eq:sigma_x123}
 \LAT_0 \cdot E =
 \{ \sigma = \sum_{i=1}^{3} x_i E_i E
 = \frac{k}{2} \SqMat{l x_2}{l x_3}{k x_1}{-l x_2}
 \bigm|
 x_1,x_2,x_3 \in \Z \}.
\end{equation}
For $\sigma$ as in (\ref{eq:sigma_x123}), we have
 $\ANG{\sigma,\sigma}_B = \OP{Nr}(\sigma)%
 = \det(\sigma) = \DDz \cdot (k x_1 x_3+l x_2^2)$,
 which is ensured by (\ref{eq:B_and_L_dotE}).
By Proposition \ref{prop:even_clifford_LAT}
 (and Corollary \ref{cor:ortho_plus}),
 the following holds:

\begin{proposition} \label{prop:ternary_mat_rep}
We have an isomorphism
\begin{equation}
 \PGB{\clfC} \CONGa 
 \{ g \in O^+(\LAT) \bigm| \bar{g}= \det(g) \cdot \id_{A(\LAT)} \}
\end{equation}
 defined by
\begin{equation}
 \alpha \mapsto \varphi_\alpha
 := (\sigma \mapsto \OP{Nr}(\alpha) \cdot \alpha \sigma \alpha^{-1}
 = \alpha \sigma \alpha^*).
\end{equation}
Here $\bar{g}$ denotes the image of $g$ in $O(q(\LAT))$ and
\begin{equation}
 \PGB{\clfC} =
 \{ \alpha =
 \SqMat{a}{b}{c}{d}
 \in PGL_2(\Z) \bigm|
 a-d \equiv c \equiv 0 \bmod{k}, ~
 b \equiv 0 \bmod{l}
 \}.
\end{equation}
Under the basis
 $(\frac{k}{2} \tilde{e}_1,-l \tilde{e}_2,\frac{k}{2} \tilde{e}_3)$
 of $\LAT_0 \cdot E=\LAT_0(\DDz)$,
 the isomorphism $\varphi_\alpha$ is represented by
 the following matrix $P_\alpha$:
\begin{equation}
 P_\alpha = \begin{pmatrix}
 a^2 & 2 a b & (-k/l) \cdot b^2 \\
 a c & a d+b c & (-k/l) \cdot b d \\
 (-l/k) \cdot c^2 & (-l/k) \cdot 2 c d & d^2
 \end{pmatrix}.
\end{equation}
In fact, we have
\begin{equation}
 {P_\alpha}^T \cdot Q_{\LAT} \cdot P_\alpha
 =(ad-bc)^2 \cdot Q_{\LAT}=Q_{\LAT}, \quad
 Q_{\LAT} =
 \begin{pmatrix}
  0 & 0 & k \\
  0 & 2l & 0 \\
  k & 0 & 0
 \end{pmatrix}.
\end{equation}
\end{proposition}
%
%
\begin{proof}
One can find the matrix $P_\alpha$ by a direct computation.
(We use $-l \tilde{e}_2$ instead of $l \tilde{e}_2$,
 in order to obtain a ``good'' matrix representation.)
\end{proof}

By (\ref{eq:Ei_Lat_kl}), we have
\begin{align}
 E_1 E_2 E_3 E
 {} &= \frac{1}{E^2} \cdot E_1 E \cdot E_2 E \cdot E_3 E \\
 {} &= - \frac{1}{\DDz} \cdot \frac{k^2 l}{4}
  \cdot \tilde{e}_3 \tilde{e}_2 \tilde{e}_1
 = \SqMat{0}{0}{0}{k^2 l/2}.
\end{align}
If we write $\beta \in \CLFM{\LAT}$ as
 $\beta = \sum_{i=1}^{3} x_i E_i + x_4 E_1 E_2 E_3$,
 then
\begin{equation}
 \CLFM{\LAT} \cdot E = \{
 \beta E
 = \frac{k}{2} \SqMat{l x_2}{l x_3}{k x_1}{-l x_2 + kl x_4}
 \bigm| x_1,\ldots,x_4 \in \Z
 \} \subset \clfC_\Q
\end{equation}
 by (\ref{eq:sigma_x123}).
Thus, by $E E^*=\DDz$,
\begin{align}
 N \beta &= \beta \beta^*
 =\DDz^{-1} \cdot \beta E \cdot (\beta E)^*
 =\DDz^{-1} \cdot \OP{Nr}(\beta E) \\
 &=k x_1 x_3 + l x_2 (x_2-k x_4). \label{eq:N_beta_value}
\end{align}
(The intersection matrix of $\CLFM{L}$
 for the general case is given in
 Remark \ref{rem:clf_minus_mat}(iii).)
Therefore, Theorem \ref{thm:MAIN} implies:
\begin{proposition} \label{prop:CL_minus_unit}
We have a bijection
\begin{equation}
 \sV / \{ \pm 1 \} \rightarrow
 \{ g \in O_\GAM(\LAT) \bigm| \det(g)=-1 \};
 \quad
 \pm \alpha \mapsto
 (\sigma \mapsto - \alpha \sigma \alpha^{-1}),
\end{equation}
 where
\begin{align}
 \sV := {} & \CLFM{\LAT}^\times \cdot E \\
 = {} & \{ \beta E \bigm| \beta \in \CLFM{\LAT}, ~
  N \beta = \pm 1 \} \\
 = {} & \{ \alpha \in \CLFM{\LAT} \cdot E \bigm|
  \OP{Nr}(\alpha) = \pm \DDz \} \\
 = {}  &
 \left\{
  \frac{k}{2} \SqMat{l x_2}{l x_3}{k x_1}{-l x_2 + kl x_4}
  \, \middle| \,
  \begin{array}{ll}
   x_1,\ldots,x_4 \in \Z, \\
   k x_1 x_3 + l x_2 (x_2-k x_4) = \pm 1
  \end{array}
 \right\}.
\end{align}
Moreover, for $\varepsilon \in \{ \pm 1 \}$,
 the following equivalence holds (see Remark \ref{rem:LAT_kl_rep}):
\begin{align}
 & N \beta=\varepsilon \ \ (\exists \beta \in \CLFM{\LAT}) \\
 &~\Leftrightarrow~
 \OP{Nr}(\alpha)=\varepsilon \cdot \DDz \ \ %
 (\exists \alpha \in \CLFM{\LAT} \cdot E) \\
 &~\Leftrightarrow~
 \sigma^2=\varepsilon \ \ (\exists \sigma \in \LAT_0).
\end{align}
In particular, $\sV = \varnothing$
 if and only if $\LAT_0$ does not represent $\pm 1$.
\end{proposition}


\begin{remark} \label{rem:LAT_kl_rep}
For $\varepsilon \in \{ \pm 1 \}$,
 the lattice $\LAT_0$ represents $\varepsilon$
 (that is, $\sigma^2=\varepsilon$ for some $\sigma \in \LAT_0$)
 if and only if
\begin{equation} \label{eq:condition_kl}
 \OP{gcd}(k,l)=1
 \text{\ \ and\ \ }
 \varepsilon \cdot l \equiv \xi^2 \bmod{k} \ \ (\exists \xi \in \Z).
\end{equation}
This is because
 $\LAT_0$ represents $\varepsilon$ if and only if
 $k xz+l y^2=\varepsilon$
 for some $x,y,z \in \Z$.
Since $N \beta$ for $\beta \in \CLFM{\LAT}$ is given by
 (\ref{eq:N_beta_value}),
 the condition (\ref{eq:condition_kl})
 is also equivalent to
 $N \beta=\varepsilon$ for some $\beta \in \CLFM{\LAT}$.
\end{remark}

We combine Propositions
 \ref{prop:ternary_mat_rep} and \ref{prop:CL_minus_unit}
 with Corollary \ref{cor:ortho_plus}
 to conclude:

\begin{proposition} \label{prop:disc_kernel_BV}
We have
\begin{equation}
 O_\GAM(\LAT)
 = \{ (\sigma \mapsto \alpha \sigma \alpha^{-1}) \bigm|
 \alpha \in \clfC^\times \}
 \sqcup
 \{ (\sigma \mapsto - \alpha \sigma \alpha^{-1}) \bigm|
 \alpha \in \sV \},
\end{equation}
 where $\sV$ is defined as in Proposition \ref{prop:CL_minus_unit}.
This decomposition is according to
 whether each element is in $SO_\GAM(\LAT)$ or not.
Moreover, suppose that
 $A(\LAT)$ is not a $2$-elementary abelian group.
Then
 the images of $\clfC^\times$ and $\sV$
 in $\clfC^\times_\Q / \Q^\times \cong PGL_2(\Q)$
 do not intersect and we have
\begin{gather}
 \{ g \in O^+(\LAT) \bigm| \bar{g}=\pm \id_{A(\LAT)} \}
 = \{ \varphi_\alpha \bigm|
 \alpha \in \clfC^\times \sqcup \, \sV \}
 \cong O_\GAM(\LAT), \\
 \det(\varphi_\alpha) = \nu(\alpha), \quad
 \bar{\varphi}_\alpha
 = \varepsilon_\alpha \cdot \nu(\alpha) \cdot \id_{A(\LAT)},
\end{gather}
 where
 $\varphi_\alpha := (\sigma \mapsto%
  \nu(\alpha) \cdot \alpha \sigma \alpha^{-1})$;
 $\varepsilon_\alpha=1$ if $\alpha \in \clfC^\times$
 and $\varepsilon_\alpha=-1$ if $\alpha \in \sV$;
 and $\nu(\alpha)$ is the sign of $\OP{Nr}(\alpha)$, or
\begin{equation*}
 \nu(\alpha) =
 \begin{cases}
  \OP{Nr}(\alpha) & \text{if} \quad \alpha \in \clfC^\times, \\
  N(\alpha E^{-1}) = \OP{Nr}(\alpha)/\DDz & \text{if} \quad \alpha \in \sV.
 \end{cases}
\end{equation*}
\end{proposition}

Combined with these results,
 Proposition \ref{prop:odd_pic_num} implies:

\begin{proposition}
Suppose that $\LAT$ has signature $(1,2)$
 (namely $l<0$) and does not represent $-2$,
 that is, $\ROOT(\LAT)=\varnothing$.
Let $X$ be a K3 surface with $S_X \cong \LAT$.
(Such a K3 surface exists by Theorem \ref{THM_Nikulin2}.)
In the same setting as in
 Proposition \ref{prop:disc_kernel_BV},
 we have
\begin{equation}
 \Aut(X) \cong
 \{ g \in O^+(\LAT) \bigm| \bar{g}=\pm \id_{A(\LAT)} \}
 = \{ \varphi_\alpha \bigm|
 \alpha \in \clfC^\times \sqcup \, \sV \}
 \cong O_\GAM(\LAT).
\end{equation}
Here
 $\sV = \{ \alpha \in \CLFM{\LAT} \cdot E \bigm| \OP{Nr}(\alpha) = \DDz \}$.
\end{proposition}

\subsection{Some isomorphic orthogonal groups}

Let $L$ and $M$ be lattices with the same signature.
In the following proposition,
 we mean by $O(L) \cong O(M)$
 that these orthogonal groups are
 not only isomorphic abstract groups, but also
 conjugate subgroups in $O(L_\R) \cong O(M_\R)$

\begin{proposition} \label{prop:N_kl}
For non-zero integers $\kappa$ and $\lambda$,
 set $\LATa_{\kappa,\lambda} := U(\kappa) \OSUM \LL{\lambda}$.
Then we have the following natural isomorphisms:
\begin{enumerate}
\item
 if $\OP{gcd}(\kappa,\lambda)=1$, then
 $O(\LATa_{\kappa,\lambda}) \cong O(\LATa_{1,\kappa \lambda})$;
\item
 if both $\kappa$ and $\lambda$ are odd, then
 $O(\LATa_{\kappa,\lambda}) \cong O(\LATa_{\kappa,4 \lambda})$;
\item
 for any $\kappa$ and $\lambda$,
 there exists a non-zero integer $\nu$
 such that $O(\LATa_{\kappa,\lambda}) \cong O(\LATa_{1,\nu})$.
\end{enumerate}
\end{proposition}
\begin{proof}
For a $\Q$-valued lattice $L$ and a non-zero integer $\mu$,
 define sublattices of $L$ by
\begin{equation*}
 \VE{L}{\mu} :=
 \{ v \in L \bigm| \ANG{v,L} \subseteq \mu \Z \}, \quad
 \mu \cdot L := \{ \mu v \bigm| v \in L \} \cong L(\mu^2).
\end{equation*}
Assume
$\OP{gcd}(\kappa,\lambda)=1$
 and set
\begin{align*}
 \LATb_1 := {} &
 \VE{\LATa_{\kappa,\lambda}}{\kappa}
 \cong U(\kappa) \OSUM \ANG{\kappa^2 \lambda}
 \cong \LATa_{1,\kappa \lambda}(\kappa), \\
 \LATb_2 := {} &
 \VE{\LATb_1}{\lambda}
 = \VE{(\VE{\LATa_{\kappa,\lambda}}{\kappa})}{\lambda}
 = (\kappa \lambda) \cdot \LATa_{\kappa,\lambda}^\vee.
\end{align*}
Thus, for $g \in O((\LATa_{\kappa,\lambda})\OverQ)$,
\begin{align*}
 & g(\LATa_{\kappa,\lambda}) = \LATa_{\kappa,\lambda}
 ~ \Rightarrow ~
 g(\LATb_1) = \LATb_1
 ~ \Rightarrow ~
 g(\LATb_2) = \LATb_2 \\
 & \Rightarrow ~
 g(\LATa_{\kappa,\lambda}^\vee) = \LATa_{\kappa,\lambda}^\vee
 ~ \Rightarrow ~
 g(\LATa_{\kappa,\lambda}) = \LATa_{\kappa,\lambda}.
\end{align*}
This implies the assertion (1).
Under the assumption of (2), let $L=\LATa_{\kappa,\lambda}$ and define
\begin{equation*}
 L' :=
 \{ v \in L \bigm| \ANG{v,v} \equiv 0 \bmod{2} \}.
\end{equation*}
Then $L'$ is an even sublattice of $L$ of index $2$
 and $L' \cong \LATa_{\kappa,4 \lambda}$.
Conversely, let $v \bmod L \in A(L')$
 be the unique element in $A(L')$ of order $2$.
Then $L$ is generated by $L'$ and $v$.
Hence $O(L) \cong O(L')$.
Thus the assertion (2) is verified.
The assertion (3) immediately follows from (1).
\end{proof}

\begin{remark}
By Proposition \ref{prop:N_kl},
 we know that
 the orthogonal group $O(\LAT)$ of our $\LAT \cong U(k) \OSUM \LL{2l}$
 is isomorphic to $O(U \OSUM \LL{2m})$ for some non-zero integer $m$.
\end{remark}

\subsection{Some special cases of $\LAT$}

Here we give a few examples of
 $\LAT \cong U(k) \OSUM \LL{2l}$ with $k$ and $l$ specialized,
 applying the results in the previous sections.
First we recall the following.

For a non-zero integer $k$, we set
\begin{equation}
 \Pi_0(k) :=
 \{
 \SqMat{a}{b}{c}{d}
 \in PGL_2(\Z) \bigm|
 c \equiv 0 \bmod{k}
 \}, \quad
 \Gamma_0(k) := \Pi_0(k) \cap \Gamma.
\end{equation}
For convenience, we admit that $k$ is negative.
Obviously we have $\Pi_0(k)=\Pi_0(|k|)$
 and $\Gamma_0(k)=\Gamma_0(|k|)$.
By Proposition \ref{prop:strong-approx},
 for an indefinite even lattice $L$ of rank $\geq 3$,
 we have
\begin{equation}
 (SO^+)^\GAMu(L) \unlhd O(q(L)), \quad
 O(q(L))/(SO^+)^\GAMu(L) \cong (\Z/2 \Z)^r,
\end{equation}
 for some $r$.
When $\rank\, L=3$ and $B \cong \CLFP{L}$,
 we have
\begin{align}
 SO_\GAM(L) \cong {} & \PGB{B}, \\ 
 SO^+_\GAM(L) \cong {} & \PSB{B}, \\ 
 SO^+(L) \cong {} & \Gamma_{\NORM^+}(B)
 := \{ \alpha \in B_\R^1 \bigm| \alpha B \alpha^{-1} = B \} / \R^\times,
\end{align}
 by Theorem \ref{thm:MAIN} and
 Remark \ref{rem:SO_normal}.
The orthogonal groups $SO^+_\GAM(L)$ and $SO^+(L)$
 are considered as discrete subgroups of
 $B_\R^1 / \R^\times \cong PSL_2(\R)$.

{\bf Notations.}
We fix notations as follows:
 $k$ is a non-zero integer.
The prime factorization of $|k|$ is given by
\begin{equation}
 |k| = \prod_{i=1}^{\nu} q_i, \quad
 q_i=p_i^{e_i}, \quad e_i \geq 1,
\end{equation}
 where the $p_i$ are distinct primes.
Moreover, we set
 $\nu' := \#\{ i \bigm| e_i \equiv 1 \bmod{2} \} \leq \nu$.

\begin{example} \label{exm:u2k}
Let $L = U \OSUM \LL{2l}$.
Then the natural map $SO^+(L) \rightarrow O(q(L))$ is surjective.
One can show this by applying Theorem \ref{thm:MAIN}, as follows.
Set
\begin{equation}
 B :=
 \{
 \SqMat{a}{b}{c}{d}
 \in M_2(\Z) \bigm|
 c \equiv 0 \bmod{l}
 \}
 \cong \CLFP{L}.
\end{equation}
(We take the transposition of $\clfC$ in (\ref{eq:clf_LAT_B}).)
Then
\begin{equation}
 SO_\GAM(L)   \cong \PGB{B} = \Pi_0(l), \quad
 SO^+_\GAM(L) \cong \PSB{B} = \Gamma_0(l).
\end{equation}
Moreover, $\sV$ as in Proposition \ref{prop:CL_minus_unit}
 is given as
\begin{equation}
 V :=
 \left\{ \alpha =
  \frac{1}{2} \SqMat{l y_1}{y_2}{l y_3}{l y_4}
  \, \middle| \,
   y_1,\ldots,y_4 \in \Z, ~
   \OP{Nr}(\alpha) = \pm \DDz = \pm \frac{-l}{4}
 \right\}.
\end{equation}
We slightly modify $V$ and define
\begin{equation}
 V' :=
 \left\{ \alpha =
  \sqrt{l_0} \SqMat{y_1}{y_2/l_0}{y_3}{y_4}
  \, \middle| \,
   y_1,\ldots,y_4 \in \Z, ~
   \OP{Nr}(\alpha) = \pm 1
 \right\}, \quad l_0 := |l|.
\end{equation}
Now Proposition \ref{prop:disc_kernel_BV} implies
\begin{equation}
 O_\GAM(L) \cong (B^\times \sqcup V')/ \{ \pm 1 \}
 = \ANG{\Pi_0(l),
  \sqrt{l_0} \left(\SqMatSmall{0}{1/l_0}{1}{0}\right)}.
\end{equation}
A direct computation shows the following:
\begin{equation}
 SO^+(L) \cong \Gamma_{\NORM^+}(B) = \Gamma^+_0(l),
\end{equation}
 where $\Gamma^+_0(l) \subset PSL_2(\R)$,
 containing $\Gamma_0(l)$ as a normal subgroup,
 is defined by
\begin{equation}
\begin{split}
 \Gamma^+_0(l) :=
 \{ \alpha
 = \sqrt{l'} \SqMat{a_0}{b_0/l'}{c_0 (l/l')}{d_0}
 \bigm| a_0,b_0,c_0,d_0 \in \Z, ~ 
 l' \in \Z_{>0}, & \\
 l' | l, ~
 \OP{Nr}(\alpha) = a_0 d_0 l'-b_0 c_0 (l/l') = 1 &
 \}.
\end{split}
\end{equation}
Obviously $a_0 d_0 l'$ and $b_0 c_0 (l/l')$
 are always coprime.
In particular, if $l' \neq 1$,
 then $b_0/l'$ is a non-zero irreducible fraction.
Hence the map
 $\alpha \mapsto \OP{rad}(l') := \prod_{p | l'} p$
 is well-defined.
(Alternatively,
 $\OP{rad}(l')$ is characterized as
 a square-free positive integer $r$
 such that $\sqrt{r} \cdot \alpha \in GL_2(\Q)$.)
Moreover, this map induces an isomorphism
\begin{equation} \label{eq:Gamma_plus_mod}
 \Gamma^+_0(l) / \Gamma_0(l)
 \CONGa
 \prod_{p | l} (p^{\Z}/p^{2 \Z}) \cong (\Z/2 \Z)^\nu.
\end{equation}
Thus we have
\begin{equation}
 (SO^+)^\GAMu(L) \cong
 SO^+(L) / SO^+_\GAM(L) \cong
 \Gamma^+_0(l) / \Gamma_0(l)
 \cong (\Z/2 \Z)^\nu.
\end{equation}
On the other hand,
 one can easily check that $O(q(L)) \cong (\Z/2\Z)^\nu$.
Hence $(SO^+)^\GAMu(L) = O(q(L))$, that is,
 the natural map $SO^+(L) \rightarrow O(q(L))$
 is surjective.
\end{example}

\begin{remark} \label{rem:spinor_u2k}
(i)
It is known that 
 $\Gamma^+_0(l)$ 
 is a maximal discrete subgroup of $PSL_2(\R)$
 if and only if $l$ is square-free.
(ii)
Let $L=U \OSUM \LL{2l}$.
The spinor norm map for $SO^+(L)$
 is given as $\alpha \mapsto l'$ in the same notations as above,
 which implies
\begin{align}
 \theta(SO^+(L))
 & {} = \{ l' \bmod{(\Q^\times)^2} \bigm| l' \in \Z_{>0}, ~
 l' | l, ~ \OP{gcd}(l',l/l')=1 \} \\
 & {} = \langle q_i \bmod{(\Q^\times)^2} \bigm| i=1,\ldots,\nu \rangle
 \cong (\Z/2\Z)^{\nu'}
\end{align}
 and
\begin{equation}
 \theta(SO(L)) {} = \langle \theta(SO^+(L)),-1 \rangle. 
\end{equation}
In addition, a similar argument for
 $L_p := L \otimes \Z_{p}$ shows
\begin{equation}
 \theta_p(SO(L_p))
 \equiv \Z_p^\times \cdot \{ 1,l \}
 \bmod{(\Q_p^\times)^2}.
\end{equation}
(iii)
In Example \ref{exm:u2k}, the surjectivity of
 $SO^+(L) \rightarrow O(q(L))$ is also proved by applying
 Proposition \ref{prop:strong-approx}
 (cf.\ Example \ref{exm:U_2e2}).
Here we give only an outline.
The codomain of the map $\gamma$
 in Proposition \ref{prop:strong-approx} is isomorphic to
\begin{equation*}
 \prod_{p|l} \frac%
 {\Z_p^\times \cdot \{ 1,l \} \bmod{(\Q_p^\times)^2}}%
 {\Z_p^\times \bmod{(\Q_p^\times)^2}}
 \cong (\Z/2\Z)^{\nu'},
\end{equation*}
 and hence $\theta(SO^+(L))$ maps onto it.
Therefore,
 $\OP{Coker}(\gamma)=0$ and the surjectivity is verified.
\end{remark}

\begin{example} \label{exm:uk2}
Let $L = U(k) \OSUM \LL{2}$.
We shall show
\begin{equation}
 [O(q(L)) : (SO^+)^\GAMu(L)] = 2^\nu.
\end{equation}
Set
\begin{equation} \label{eq:B_for_uk2}
 B :=
 \{
 \SqMat{a}{b}{c}{d}
 \in M_2(\Z) \bigm|
 a-d \equiv c \equiv 0 \bmod{k}
 \} \cong \CLFP{L}.
\end{equation}
Then $SO_\GAM(L) \cong \Gamma^\times(B)$.
Moreover, $\sV$ as in Proposition \ref{prop:CL_minus_unit}
 is given as
\begin{equation}
 V := \left\{ \frac{k}{2} \cdot \alpha \, \middle| \,
  \alpha =
  \SqMat{x_2}{x_3}{k x_1}{-x_2 + k x_4}, ~
   x_1,\ldots,x_4 \in \Z, ~
   \OP{Nr}(\alpha) = \pm 1
 \right\}.
\end{equation}
Hence
\begin{align}
 O_\GAM(L) & {} \cong
 \{
 \SqMat{a}{b}{c}{d}
 \in PGL_2(\Z) \bigm|
 a - \mu d \equiv c \equiv 0 \bmod{k} ~
 (\exists \mu \in \{ \pm 1 \})
 \} \\
 & = \ANG{ \Gamma^\times(B) ,
  \left(\SqMatSmall{1}{0}{0}{-1}\right) }.
\end{align}
Suppose that $k$ is odd.
Then
 $SO^+(L) \cong SO^+(U \OSUM \LL{2k)}$ 
 by Proposition \ref{prop:N_kl}(1).
We have
 $SO^+(L) \cong \Gamma_{\NORM^+}(B) = \Gamma^+_0(k)$.
Hence
\begin{align*}
 |(SO^+)^\GAMu(L)| &= [SO^+(L) : SO^+_\GAM(L)] =
 [\Gamma^+_0(k) : \PSB{B}] \\
 &= [\Gamma^+_0(k) : \Gamma_0(k)] \cdot [\Gamma_0(k) : \PSB{B}] \\
 &= 2^\nu \cdot (2^{-\nu} \cdot \varphi(k))
 = \varphi(k),
\end{align*}
 where 
 $\varphi(k) := |(\Z/k \Z)^\times|$ is Euler's totient function
 and we use
\begin{equation*}
 \Gamma_0(k) / \PSB{B}
 \cong \frac{
  (\Z/k \Z)^\times
 }{
  \{ \bar{a} \in (\Z/k \Z)^\times \bigm|
  a^2 \equiv 1 \bmod{k} \}
 }, \quad
 |\Gamma_0(k) / \PSB{B}| = 2^{-\nu} \cdot \varphi(k).
\end{equation*}
On the other hand, one can also verify
 $|O(q(L))| = 2^\nu \cdot \varphi(k)$ directly.
Hence the discriminant image $(SO^+)^\GAMu(L)$
 has index $2^\nu$ in $O(q(L))$.
(Similarly to Remark \ref{rem:spinor_u2k},
 we can apply Proposition \ref{prop:strong-approx}
 to this case, as well.
The codomain of $\gamma$ is isomorphic to
 $(\Z/2\Z)^{\nu+\nu'}$
 and $\theta(SO(L)) \cong (\Z/2\Z)^{\nu'}$
 maps injectively into it,
 and thus $\OP{Coker}(\gamma) \cong (\Z/2\Z)^{\nu}$.)
\end{example}

\begin{example} \label{exm:U_2e2}
Let $L=U(2^e) \OSUM \LL{2}$ with $e \geq 1$.
We shall show:
\begin{equation}
 [O(q(L)) : (SO^+)^\GAMu(L)] =
 \begin{cases}
  1 & \quad (e=1) \\
  2 & \quad (e=2) \\
  4 & \quad (e \geq 3)
 \end{cases}
\end{equation}
We omit the proof for $e=1$, $2$.
Assume $e \geq 3$.
Then
\begin{equation} \label{eq:exm_3_isom}
 SO^+(L) \cong SO^+(U(2^{e-1}) \OSUM \LL{1)} \cong
 SO^+(U \OSUM (2^{e-1})) \cong \Gamma^+_0(2^{e-2})
\end{equation}
 by Proposition \ref{prop:N_kl}(1)
 and Example \ref{exm:u2k}.
We use the same $B$ as in (\ref{eq:B_for_uk2}) with $k=2^e$.
Since $\Gamma^\times(B) = \Gamma^1(B)$, it follows that
\begin{equation}
 SO_\GAM(L) = SO_\GAM^+(L).
\end{equation}
Set $L_2 := L \otimes \Z_2$ and $B_2 := B \otimes \Z_2$.
By (\ref{eq:exm_3_isom}) and Remark \ref{rem:spinor_u2k}, we have
\begin{equation*}
 \theta(SO^+(L)) {} 
 \equiv \{ 1,2^e \} \bmod{(\Q^\times)^2}, \quad 
 \theta_2(SO(L_2))
 \equiv \Z_2^\times \cdot \{ 1,2^e \}
 \bmod{(\Q_2^\times)^2}.
\end{equation*}
Since $SO_\GAM(L_2) \cong \PGB{B_2}$, it  follows that
\begin{equation*}
 \theta_2(SO_\GAM(L_2))
 \equiv \OP{Nr}(B_2^\times)
 \equiv \{ 1 \} \bmod{(\Q_2^\times)^2}.
\end{equation*}
Now we apply Proposition \ref{prop:strong-approx}.
Since $\det(O_\GAM(L_2)) = \{ \pm 1 \}$, we have
\begin{equation*}
 H_2 :=
 \frac{f_2(O(L_2))}{f_2(O_\GAM(L_2))} \cong
 \frac{\theta_2(SO(L_2))}{\theta_2(SO_\GAM(L_2))}
 \cong \theta_2(SO(L_2)).
\end{equation*}
Hence the cokernel of the map $\gamma$
 as in Proposition \ref{prop:strong-approx}
 is isomorphic to
\begin{equation*}
 \frac{H_2}{f_2(SO^+(L))} \cong
 \frac{\theta_2(SO(L_2))}{\theta_2(SO^+(L))} \cong
 \frac{\Z_2^\times \cdot \{ 1,2^e \} \bmod{(\Q_2^\times)^2}}%
      {\{ 1,{\displaystyle 2^e} \} \bmod{(\Q_2^\times)^2}}
 \cong \Z_2^\times \cong (\Z/2\Z)^2.
\end{equation*}
Thus the index $[O(q(L)) : (SO^+)^\GAMu(L)]$ is equal to $4$.
Alternatively, this index can be directly verified
 by showing the following (we omit the details):
\begin{align}
 |O(q(L))| & = 2^{e+2}, \\
 \Gamma_{\NORM^+}(B)
 & = \rho \Gamma^+_0(2^{e-2}) \rho^{-1}, \quad
 \rho := \left( \SqMatSmall{1}{0}{0}{2} \right), \\
 |(SO^+)^\GAMu(L)| & =
 [SO^+(L) : SO^+_\GAM(L)]
 = [\rho \Gamma^+_0(2^{e-2}) \rho^{-1} : \PSB{B}]=2^e.
\end{align}
In addition, we have $[SO^\GAMu(L) : (SO^+)^\GAMu(L)]=2$
 and $O^\GAMu(L) = SO^\GAMu(L)$.
\end{example}

\section{Example: $U(n) \OSUM \LL{-2n}$} \label{sect:lat_u_n_2n}

In this section,
 we apply the results in Section \ref{subsect:U_kl}
 to the following lattice $M_n$
 for a positive integer $n$:
\begin{equation}
 M :=
 \begin{pmatrix} 0& 0&1\\0&-2&0\\ 1&0&0 \end{pmatrix}
 \cong U \OSUM \LL{-2}, \quad
 M_n := M(n) \cong U(n) \OSUM \LL{-2n}.
\end{equation}
We identify $O(M_n)$ for each $n$ with $O(M)$ in a natural way.
Under the same notations for congruence subgroups
 as in Section \ref{subsect:cong}, we have
\begin{equation}
 O(M)=O_\GAM(M) \cong \Pi \times \{ \pm \id_{M} \}, ~
 SO(M) \cong O^+(M) \cong \Pi, ~
 SO^+(M) \cong \Gamma.
\end{equation}
Applying Proposition \ref{prop:ternary_mat_rep}
 with $k=1$ and $l=-1$,
 we obtain an isomorphism
\begin{equation}
 \Pi \cong O^+(M); \quad
 \alpha = \SqMat{a}{b}{c}{d} \leftrightarrow
 P_\alpha = \begin{pmatrix}
 a^2 & 2 a b & b^2 \\
 a c & a d+b c & b d \\
 c^2 & 2 c d & d^2
 \end{pmatrix}.
\end{equation}
The correspondence
 $\alpha \leftrightarrow \OP{Nr}(\alpha) \cdot P_\alpha$
 gives an isomorphism $\Pi \cong SO(M)$.

\subsection{Automorphism group of $X_n$}

Now we study automorphism groups of
 a certain kind of K3 surfaces.

\begin{proposition} \label{prop:isometry-of-L}
For $n \geq 2$, the following hold:
\begin{gather}
 B_n := \{ \alpha \in M_2(\Z) \bigm|
  \alpha \equiv \lambda I \bmod{n} ~
  (\exists \lambda \in \Z)
 \} \cong \CLFP{M_n}, \\
 O_\GAM(M_n) = SO_\GAM(M_n) \cong G_n, \quad
 SO^+_\GAM(M_n) \cong G_n \cap \Gamma,
\end{gather}
 where
\begin{equation}
 G_n := \{ \alpha \in \Pi \bigm|
  \alpha \equiv \lambda I \bmod{n} ~
  (\exists \lambda \in \Z)
 \}.
\end{equation}
\end{proposition}

\begin{proposition}
For $n \geq 2$,
 let $X_n$ be an arbitrary K3 surface whose Picard lattice is
 isomorphic to $M_n$.
Then
\begin{equation}
 \Aut(X_n) \cong
 O_\GAM(M_n) = SO_\GAM(M_n) \cong G_n.
\end{equation}
Moreover, the group of symplectic automorphisms corresponds to
 $G_n \cap \Gamma$.
There exists an anti-symplectic automorphism of $X_n$
 if and only if
 $-1$ is a quadratic residue modulo $n$
 (see Lemma \ref{lem:condition_negative_det}).
\end{proposition}

We provide here some facts on $M_n$.
First, for a ring $R$, define
\begin{align}
 PGL_2(R) := {} &
 \{ A \in M_2(R) \bigm| \det(A) \in R^\times \} / R^\times, \\
 PSL_2(R) := {} &
 \{ A \in PGL_2(R) \bigm| \det(A) = 1 \}.
\end{align}
We have a natural isomorphism
 $PGL_2(R) / PSL_2(R) \cong R^\times / (R^\times)^2$.
For simplicity, we write
 $\sM := M_n$ and $\sB := B_n$.
For each prime $p$ dividing $n$, set
 $\sM_p := \sM \otimes \Z_p$ and $\sB_p := \sB \otimes \Z_p$.
Let $q=p^e || n$, that is, $q|n$ and $\OP{gcd}(q,n/q)=1$.
Our computations in the previous sections are valid over $\Z_p$
 and we obtain
\begin{gather}
 O(\sM_p) \cong \Pi_p \times \{ \pm \id \}, \quad
  \Pi_p := PGL_2(\Z_p), \\ 
 O_\GAM(\sM_p) = SO_\GAM(\sM_p) \cong
 \Gamma^\times(\sB_p) := \sB_p^\times / \Z_p^\times \subseteq \Pi_p, \\
 O(q(\sM_p)) \cong O^\GAMu(\sM_p) \cong
  O(\sM_p) / O_\GAM(\sM_p) \cong
  PGL_2(\Z/q\Z) \times \{ \pm \id \}.
\end{gather}
The last line follows from
 the surjectivity of the natural map
 $O(\sM_p) \rightarrow O(q(\sM_p))$,
 which holds in general \cite[Corollary 1.9.6]{N}.
Hence
\begin{equation}
 O(q(\sM)) \cong \prod_{p | n} O(q(\sM_p)) \cong
 PGL_2(\Z / n \Z)
 \times
 \prod_{p | n} \{ \pm \id \in O(q(\sM_p)) \}.
\end{equation}
We have $\Gamma(n) \unlhd G_n$ and
\begin{equation}
 G_n / \Gamma(n)
 \cong \{ \lambda \in \Z/n \Z \bigm|
 \lambda^2 \equiv \pm 1 \bmod{n} \} / \{ \pm 1 \bmod{n} \}.
\end{equation}
The group $G_n$ is contained in $\Gamma$ if and only if
 $-1$ is a quadratic residue modulo $n$.
Since
\begin{equation}
 (SO^+)^\GAMu(\sM) \cong {} SO^+(\sM) / SO_\GAM^+(\sM) \cong
 \Gamma / (G_n \cap \Gamma)
 \cong {} PSL_2(\Z / n \Z),
\end{equation}
 it follows that
\begin{equation}
 \frac{O(q(\sM))}{O^\GAMu(\sM)} \cong
 \frac{
  (\Z/n\Z)^\times / ((\Z/n\Z)^\times)^2 \times
  \prod_{p | n} \{ \pm \id \in O(q(\sM_p)) \}
 }{
  \{ \pm 1 \} \times \{ \pm \id \in O(q(\sM)) \}
 }.
\end{equation}
Alternatively,
 one can apply Proposition \ref{prop:strong-approx}
 to show this.

\begin{example} \label{exm:2-power-gamma}
We have 
 $G_2=\Pi(2)$ and $G_4=\Gamma(4)$.
If $n=2^e$ with $e \geq 3$, then
\begin{equation}
 \Gamma(2^e)\subsetneq G_{2^{\scriptstyle e}} 
 \subsetneq \Gamma(2^{e-1}), \quad
 G_{2^{\scriptstyle e}} / \Gamma(2^e) \cong \Z/2 \Z, \quad
 \Gamma(2^{e-1}) / \Gamma(2^e) \cong (\Z/2 \Z)^3.
\end{equation}
By Proposition \ref{prop:torsion-grp},
 $G_{2^{\scriptstyle e}}$ for $e \geq 2$
 is a free group.
Moreover, we have
\begin{equation}
 [\Gamma : G_{2^{\scriptstyle e}}] =
 \begin{cases}
  24 & \text{if} \quad e =2, \\
  3 \cdot 2^{3e-4} & \text{if} \quad e \geq 3.
 \end{cases}
\end{equation}
Hence, by Proposition \ref{prop:free-group-12},
 the rank of $G_{2^{\scriptstyle e}}$ as a free group is given by
\begin{equation}
 r :=
 \rank\, G_{2^{\scriptstyle e}} =
 \frac{1}{6} \, [\Gamma : G_{2^{\scriptstyle e}}] + 1 =
 \begin{cases}
  5 & \text{if} \quad e =2, \\
  2^{3e-5} + 1 & \text{if} \quad e \geq 3.
 \end{cases}
\end{equation}
For a group $H$, we mean by $H^k$
 the (normal) subgroup generated by the elements of the form
 $g^k$ for $g \in H$.
We have $\Pi(2)^2=\Pi(4)=\Gamma(4)$ (see \cite{SI}),
 that is, $(G_2)^2=G_4$.
For $e \geq 2$, we have (see Remark \ref{rem:kth_power_G} below)
\begin{equation}
 (G_{2^{\scriptstyle e}})^2 \subsetneq G_{2^{\scriptstyle e+1}}, \quad
 G_{2^{\scriptstyle e}} / (G_{2^{\scriptstyle e}})^2 \cong (\Z/2 \Z)^r.
\end{equation}
\end{example}


\begin{remark} \label{rem:kth_power_G}
In general,
 if $k|n$, then any $g \in G_n$ satisfies $g^k \in G_{nk}$
 by Lemma \ref{lem:lattice_power}(2).
This statement also follows from the isomorphism
\begin{equation}
 (\Z/k \Z)^{\oplus 3}
 \cong \Pi(n)/\Pi(nk); \quad
 (\bar{a},\bar{b},\bar{c}) \leftrightarrow
 \SqMat{1+na}{nb}{nc}{1-na} \bmod{nk}.
\end{equation}
\end{remark}

\subsection{Abelian surfaces}

We can realize $M_n$ as the Picard lattice of
 an abelian surface.
Take a very general elliptic curve $E=\C/(\Z + \Z \tau)$ and set
\begin{equation}
 A_n := (E \times E)/\ANG{(1/n,\tau/n)}.
\end{equation}
Define curves $C'_i$ on $E \times E$ by
\begin{equation}
 C'_1 := \text{pt} \times E, ~
 C'_2 := E \times \text{pt}, ~
 C'_3 := \{ (x,x) \in E \times E \},
\end{equation}
 and let $C_i$ be the image of $C'_i$ in $A_n$.
Then the Picard lattice $S_{A_n}$ of $A_n$ is generated by
 the classes $[C_i]$.
From this, one can verify $S_{A_n} \cong M_n$.
It is straightforward to check the following isomorphism
\begin{equation}
 \Aut(A_n) \cong
 \{ \alpha \in GL_2(\Z) \bigm|
  \alpha \equiv \lambda I \bmod{n} ~
  (\exists \lambda \in \Z)
 \}; \quad
 \psi_\alpha \leftrightarrow \alpha,
\end{equation}
 where $\psi_\alpha$ is induced by
 a linear transformation by $\alpha$ in
 the coordinate system $(x,y) \in E \times E$.
We apply the Kummer construction to $A_n$.
Namely, we obtain a K3 surface $Y_n$
 as the minimal resolution of the quotient surface
 $A_n/ \{ \pm 1 \}$.
The Picard lattice $S_{A_n}$ maps to a sublattice of $S_{Y_n}$,
 which is isomorphic to $S_{A_n}(2) \cong M_{2n}$.
We have
\begin{equation}
 G_n \cong \Aut(A_n)/ \{ \pm \id \} \subseteq \Aut(Y_n).
\end{equation}
Each linear system $|C_i|$ gives an elliptic fibration
 $Y_n \rightarrow \BP^1$.
The $|C_i|$ induce a generically one-to-one map
\begin{equation}
 |C_1| \times |C_2| \times |C_3| \colon
  Y_n \rightarrow \BP^1 \times \BP^1 \times \BP^1,
\end{equation}
 whose image $\overline{Y}_n$
 is a Wehler K3 surface (see Section \ref{sect:intro}) if $n=1$ and
 a non-normal hypersurface of tri-degree $(2n,2n,2n)$ if $n \geq 2$.
A generic fiber of each projection
 $\overline{Y}_n \rightarrow \BP^1$ is a curve in $\BP^1 \times \BP^1$
 of bi-degree $(2n,2n)$
 with $4n(n-1)$ nodes,
 whose minimal resolution is an elliptic curve.

\section{Salem numbers} \label{sect:salem}

A complex number $\lambda$ is an algebraic integer
 if $S(\lambda)=0$ for some irreducible monic polynomial
 $S(x)\in \Z[x]$.
An algebraic integer $\lambda$ is called a unit if
 $\lambda^{-1}$ is also an algebraic integer.
A Salem number is a unit $\lambda>1$
 whose conjugates other than $\lambda^{\pm1}$
 lie on the unit circle.
The irreducible monic integer polynomial $S(x)$ 
 is a Salem polynomial and the degree of $S(x)$
 is the degree of the Salem number $\lambda$.
(Here we admit that $\lambda$ has degree $2$.)

Let $X$ be a K3 surface.
For an automorphism $f\in \Aut(X)$,
 it is known that either all eigenvalues of $f^*| H^2(X,\Z)$
 are roots of unity
 or there is a unique simple eigenvalue $\lambda$
 which is a Salem number \cite[Theorem 3.2]{Mc}.
Hence the irreducible factors of
 the characteristic polynomial of $f^*| H^2(X,\Z)$
 include at most one Salem polynomial;
 and the remaining factors are cyclotomic \cite[Corollary 3.3]{Mc}.

Suppose that the Picard lattice of $X$ is isomorphic to
 $L$ with signature $(1,2)$ and $\ROOT(L) = \varnothing$.
For $\varphi_\alpha \in O^+(L)$ as in Proposition \ref{cor:ortho_plus},
 its characteristic polynomial is given as
\begin{equation}
 \det(tI_3-\varphi_\alpha)
 =(t-\OP{Nr}(\alpha))(t^2- \OP{Tr}(\alpha^2) t+1), \quad
 \OP{Tr}(\alpha^2) = \OP{Tr}(\alpha)^2 -2 \OP{Nr}(\alpha).
\end{equation}
(For $\alpha \in B$, we diagonalize $\alpha$
 as an element in $B_\R \cong M_2(\R)$ and
 this equality can be checked directly.)

For example, if $L=U(2) \bot \LL{-4}$,
 then $\Aut(X) \cong O^+(L) \cong \Pi(2)$.
For $\alpha = \left( \SqMatSmall{a}{b}{c}{d} \right) \in \Pi(2)$,
 we have $\OP{Tr}(\alpha)=a+d \equiv 2 \bmod{4}$
 if $\OP{Nr}(\alpha)=1$
 and $\OP{Tr}(\alpha) \equiv 0 \bmod{4}$ if $\OP{Nr}(\alpha)=-1$.
Conversely, for 
\begin{equation}
 \beta=\begin{pmatrix}  1&2\\2n& 4n\pm 1 \end{pmatrix}\in \Pi(2)
 \qquad
 (n \in \Z),
\end{equation}
 we have $\OP{Nr}(\beta)=\pm 1$ and $\OP{Tr}(\beta)=4n+2$ or $4n$.
Therefore, the Salem polynomial of $g \in \Aut(X)$
 is of the form $t^2-At+1$, where
\begin{equation}
 A=\begin{cases}
  (4n+2)^2-2 &\text{ if $g$ is symplectic}, \\
  (4n)^2+2 &\text{ if $g$ is anti-symplectic},
 \end{cases}
\end{equation}
 for some $n \in \Z_{\geq 1}$.

\section{Appendix for Clifford algebras}

Let $L$ be an even lattice of rank $n$
 and define $L_0 := L(1/2)$.
Let $\ANGz{~,~}$ denote the bilinear form of $L_0$.
We define $Z^+$ and $Z^-$ by
\begin{equation}
 Z^\pm := \{ z \in \CLF{L}\OverQ \bigm|%
 zx=\pm xz ~~ (\forall x \in \CLF{L}\OverQ) \}.
\end{equation}
Fix an orthogonal basis $(t_i)$ of $L\OverQ$
 and set $t := t_1 \cdots t_n$.
Then, we have $Z^+=\Q$ and $Z^-=\Q \, t$ if $n$ is even,
 and $Z^+=\Q \oplus \Q \, t$ and $Z^-=0$ if $n$ is odd.

\begin{proposition} \label{prop:alternating_sum}
We use the same notations as above.
For a (non-degenerate) even lattice $L$ of rank $n$
 with a $\Z$-basis $(E_i)$,
 define elements $E$ and $\hat{E}_j$ in $\CLF{L}\OverQ$ by
\begin{align}
 & E {} := \frac{1}{n!} \sum_{\sigma \in S_n}
 \OP{sgn}(\sigma) \cdot E_{\sigma(1)} \cdots E_{\sigma(n)}, \\
 & \hat{E}_j {} := \frac{(-1)^{j+1}}{(n-1)!} \, \times
 \sum_{\sigma \in S_n; ~ \sigma(j)=j} \OP{sgn}(\sigma) \cdot
 E_{\sigma(1)} \cdots E_{\sigma(j-1)} E_{\sigma(j+1)} \cdots E_{\sigma(n)},
\end{align}
 where $S_n$ is the symmetric group of degree $n$.
Then $E$ is independent of the choice of a $\Z$-basis $(E_i)$ of $L$,
 up to sign, and we have
\begin{gather}
 E \in \Q \, t, \\
 E^* = (-1)^{n(n-1)/2} \cdot E, \quad
 E \CLFinvo{E}= \disc(L_0), \quad
 E^2 = (-1)^{n(n-1)/2} \cdot \disc(L_0), \\
 E_i E=\sum_{j=1}^{n} \ANG{E_i,E_j}_0 \cdot \hat{E}_j.
\end{gather}
Since the dual basis $(\dE_i)$ of $L_0^\vee$ 
 is given by
 $(\dE_1,\ldots,\dE_n)=(E_1,\ldots,E_n) \, Q_{L_0}^{-1}$,
 we have $\dE_i E=\hat{E}_i$, and hence
 $(\hat{E}_1 E,\ldots,\hat{E}_n E)%
 =(E_1,\ldots,E_n) \, Q_{L_0}^{-1} \cdot E^2$.
\end{proposition}
\begin{proof}
For the reader's convenience, we give (an outline of) a proof.
The assertion is trivial if $(E_i)$ is an orthogonal basis.
We work over $\Q$.
It is enough to show that the change of a $\Q$-basis $(E_i)$
 by any elementary transformation does not affect the assertion.

If we transform $(E_i)$ by $\varepsilon \in GL_n(\Q)$,
 then $E$ is multiplied by $\det(\varepsilon)$,
 by the following argument.
We consider only the case where $n=3$ and
 $\varepsilon \colon E_1 \mapsto E_1 + \lambda E_2$
 (an elementary transformation)
 and omit the other computations.
In this case, the (induced) action of $\varepsilon$ is given as, for example,
\begin{align*}
 E_1 E_2 E_3 &\mapsto
 (E_1+\lambda E_2) E_2 E_3
 =E_1 E_2 E_3+\lambda E_2 E_2 E_3, \\
 E_2 E_1 E_3 &\mapsto
 E_2 (E_1+\lambda E_2) E_3
 =E_2 E_1 E_3+\lambda E_2 E_2 E_3.
\end{align*}
Thus the term $E_2 E_2 E_3$ is cancelled out and does not appear in $E$.
Other computations are similar and
 one can conclude that $\varepsilon E = E$.
We can also check that the transformation of $(\hat{E}_j)$ by $\varepsilon$
 is given as $\hat{E}_2 \mapsto \hat{E}_2 - \lambda \hat{E}_1$
 (corresponding to the cofactor matrix of $\varepsilon$).

Under the same setting as above,
 if the equality
\begin{equation*}
 E_i E=\sum_{j=1}^{3} \ANG{E_i,E_j}_0 \cdot \hat{E}_j
 \quad (i=1,2,3)
\end{equation*}
 holds for a basis $(E_i)$, then we have
\begin{align*}
 & \sum_{j=1}^{3} \ANG{\varepsilon E_1,\varepsilon E_j}_0 \cdot
    (\varepsilon \hat{E}_j).\\
 & {} = \ANG{\varepsilon E_1,E_1+\lambda E_2}_0 \cdot \hat{E}_1
  + \ANG{\varepsilon E_1,E_2}_0 \cdot (-\lambda \hat{E}_1+\hat{E}_2)
  + \ANG{\varepsilon E_1,E_3}_0 \cdot \hat{E}_3 \\
 & {} = \sum_{j=1}^{3} \ANG{\varepsilon E_1,E_j}_0 \cdot \hat{E}_j \\
 & {} = \sum_{j=1}^{3} \ANG{E_1,E_j}_0 \cdot \hat{E}_j
      + \lambda \sum_{j=1}^{3} \ANG{E_2,E_j}_0 \cdot \hat{E}_j \\
 & {} = E_1 E+\lambda E_2 E = (\varepsilon E_1) \cdot E.
\end{align*}
We omit the computations for
 $(\varepsilon E_i) \cdot E$ for $i=2,3$.

From these computations,
 one can conclude that
 if the assertion of the proposition holds for $(E_i)$,
 it holds for $(\varepsilon E_i)$ as well.
Computations for other elementary transformations are similar and omitted.
\end{proof}

By Proposition \ref{prop:alternating_sum},
 $\hat{L}_0 := \bigoplus_{i=1}^{n} \Z \hat{E}_i = L_0^\vee \cdot E$
 does not depend on a $\Z$-basis $(E_i)$ of $L_0$.
Under the following pairing $[~,~]_E$,
 the bases $(E_i)$ and $(\hat{E}_i)$
 are dual to each other,
 that is, $[E_i,\hat{E}_j]_E = \delta_{ij}$:
\begin{equation}
  L_0 \times \hat{L}_0 \rightarrow \Z; \quad
  (x,y) \mapsto [x,y]_E :=
  \frac{\pi(xy)}{\pi(E)}
  =\frac{\pi(xy)}{E_1 \wedge \cdots \wedge E_n}.
\end{equation}
Here $\pi$ is (the $\Q$-linear extension of)
 the following natural projection:
\begin{equation}
 \pi \colon \CLF{L} \rightarrow
 \frac{\CLF{L}}{\text{the image of $\bigoplus_{k=0}^{n-1} L^{\otimes k}$}}
 \cong \wedge^n L.
\end{equation}
This pairing is compatible with
 the lattice structures of $L_0$.
Namely, we have a canonical isomorphism
\begin{equation}
 L_0^\vee \cong \Hom(L_0,\Z) \cong \hat{L}_0; \quad
 E_i^\vee \leftrightarrow
 \ANGz{-,E_i^\vee}=[-,\hat{E}_i]_E \leftrightarrow \hat{E}_i.
\end{equation}
This map is also represented as $x \mapsto x E$.
It becomes an isomorphism of lattices
 if we equip $\hat{L}_0$ with the quadratic form
 $x \mapsto \disc(L_0)^{-1} \cdot x \CLFinvo{x}$.

The pairing $[~,~]_E$ has the ambiguity of sign,
 depending on $(E_i)$ (or $E$).
We can also write it as
 $[x,y]_E=\ANG{x,y \cdot E^{-1}}_0$.

\section{Appendix for orthogonal groups} \label{sect:appendix_ortho}

\subsection{Index $[\Gamma^\times(B) : \Gamma^1(B)]$} \label{subsect:index_1_2}
Set $k := [SO_\GAM(L) : SO^+_\GAM(L)] = [\Gamma^\times(B) : \Gamma^1(B)]$
 in the same notations as in Corollary \ref{cor:clifford_plus}.
If $L$ is definite, then $k=1$.
On the other hand, if $L$ is indefinite,
 then $k=1$ or $2$, where each case can occur as follows.
Set $O^-(L) := O(L) \setminus O^+(L)$.

(1)
Suppose that a sublattice $M \subset L$
 is of rank $2$ and indefinite, and that the negative Pell equation
 $x^2-D y^2=-4$ has a solution $(x,y) \in \Z^2$,
 where $D := -\OP{disc}(M) > 0$.
Then, by \cite{HKL},
 there exits $g \in SO^+(M)$ such that $-g$ acts on $A(M)$ trivially.
Hence $-g$ extends to an element in $SO_\GAM(L) \cap O^-(L)$.
Therefore the existence of such $M$ implies $k=2$.
A typical example satisfying this condition is
 $L=M \OSUM \LL{2a}$,
 where $M$ has discriminant $=-1,-4,-5,-8,-13$, etc, and
 $a$ is any non-zero integer.
Remark: If $L$ is indefinite and has an orthogonal basis,
 then $[SO(L) : SO^+(L)]=2$, although this does not ensure $k=2$.

(2)
Set $L_p := L \otimes \Z_p$ and $B_p := B \otimes \Z_p$.
Consider
 $L=\left( \begin{smallmatrix} 4&1 \\ 1&16 \end{smallmatrix} \right)%
 \OSUM \LL{-2 \cdot 3 \cdot 7^2}$ with discriminant
 $=-2 \cdot 3^3 \cdot 7^3$.
We assume here some basics on the number theory for
 lattices (or quadratic forms) and quaternion algebras
 (see e.g.\ \cite{O,vigneras80}).
Since the prime divisors of
 $\det(2 \cdot Q_B)=3^6 \cdot 7^6$
 are only $3$ and $7$,
 we have $B_{p} \cong M_2(\Z_p)$,
 and hence $SO(L_{p}) \cong PGL_2(\Z_p)$,
 for $p \neq 3,7$.
On the other hand, one has
\begin{equation*}
 L_{3} \cong \LL{1} \OSUM \LL{3} \OSUM \LL{3^2}, \quad
 L_{7} \cong \LL{1} \OSUM \LL{7} \OSUM \LL{7^2}.
\end{equation*}
Assume that there is an element $g \in SO(L) \cap O^-(L)$.
Let $\sigma=\theta(g)$ be the spinor norm of $g$.
One can verify
\begin{equation*}
 \sigma \equiv 1,3 ~ \bmod{(\Q_3^\times)^2}, \quad
 \sigma \equiv 1,7 ~ \bmod{(\Q_7^\times)^2}.
\end{equation*}
For $p \neq 3,7$, since $SO(L_{p}) \cong PGL_2(\Z_p)$,
 it follows that
 $\sigma \in \Z_p^\times \cdot (\Q_p^\times)^2$.
Hence
\begin{equation*}
 \sigma \equiv -1,-3,-7,-21 ~ \bmod{(\Q^\times)^2},
\end{equation*}
 because $g \in SO^-(L)$.
This is a contradiction
 and one finds $SO(L) = SO^+(L)$.
In particular, we have $k=1$.

\subsection{Example of an element in $\CLFM{L}^\times$}
\label{subsect:appendix_exm}


Let $L_0 = \LL{3} \OSUM \LL{-5} \OSUM \LL{-9}$.
Then $L_0$ does not represent $\pm 1$.
However, $\CLFM{L} \cong L_0 \OSUM \LL{135}$ represents 1 by
\begin{equation*}
 \alpha := 5 E_2+E_3+E_1 E_2 E_3 \in \CLFM{L}, \quad
 N \alpha = 5^2 \cdot (-5)-9+135 = 1.
\end{equation*}
Moreover, a direct computation shows that
 $h_\alpha := (v \mapsto - \alpha v \alpha^{-1}) \in O_\GAM(L)$
 is represented by the following matrix $P_\alpha$:
\begin{equation*}
 P_\alpha =
 \begin{pmatrix}
  -269 & 90 & -450\\
  54 & -19 & 90\\
  -150 & 50 & -251
 \end{pmatrix}, ~
 Q=
 \begin{pmatrix}
  3 & 0 & 0\\
  0 & -5 & 0\\
  0 & 0 & -9
 \end{pmatrix}, ~
 {P_\alpha}^T \cdot Q \cdot P_\alpha = Q.
\end{equation*}
On the other hand, $\CLFM{L}$ does  not represent $-1$
 (even over $\Z_3$).

\section{Appendix for $\CLFM{L}$ as a lattice}

Throughout this section, we suppose that
 $L$ is an even lattice of rank $3$
 and use the results in Section \ref{sect:cliff} freely.
We regard $L^\flat := \CLFM{L}$
 as an even lattice with the quadratic form
 $2 \cdot N v=2v v^*$ ($v \in \LatM$).
(See Remark \ref{rem:clf_minus_mat}(iii)
 for the intersection matrix of $\CLFM{L}$.)
Then $\LatM$ contains $L$ as a primitive sublattice.
Define $\xi_\alpha \in O(\LatM_\Q)$
 for $\alpha \in \CLFPM{L}_\Q^\times$ by
\begin{equation}
 \xi_\alpha \colon \LatM_\Q \rightarrow \LatM_\Q; \quad
 v \mapsto
 \frac{-1}{N \alpha} \cdot \alpha v^* \alpha.
\end{equation}
The reflection of $\LatM_\Q$
 with respect to $r \in \LatM_\Q$ with $N r \neq 0$
 is represented as $\xi_r$.
In fact, we have $\xi_r(r)=-r$ and, for $v \in r^\bot$,
\begin{equation*}
 v r^* + r v^* = 0, \quad
 \xi_r(v) = v.
\end{equation*}
Since $O(L^\flat_\Q)$ is generated by reflections,
 we have an isomorphism
\begin{equation}
 \frac{
  \{ (\alpha,\beta) \in \CLFP{L}_\Q^\times \times \CLFP{L}_\Q^\times
  \bigm| N \alpha = N \beta \}
 }{\{ (c,c) \bigm| c \in \Q^\times \}} \rtimes \ANG{\tau}
 \CONGa O(L^\flat_\Q),
\end{equation}
 where $(\alpha,\beta)$ corresponds to
\begin{equation}
 \Phi_{\alpha,\beta} \colon \LatM \rightarrow \LatM; \quad
 v \mapsto \alpha v \beta^{-1},
\end{equation}
 and $\tau := (v \mapsto v^*)$.

\begin{theorem}
In the same setting as above, we have
\begin{align}
 O_\GAM(L^\flat) = {} &
 \{ \Phi_{\alpha,\beta} \bigm|
 \alpha,\beta \in \CLFP{L}^\times, ~
 \alpha \beta^{-1} \in \Xi \} \\
 & {} \sqcup
 \{ \Phi_{\alpha,\beta} \circ \tau \bigm|
 \alpha,\beta \in \CLFM{L}^\times, ~
 - \alpha \beta^{-1} \in \Xi \},
\end{align}
 where $\Xi$ is a subset of $\CLFP{L}^\times$ defined by
\begin{equation}
 \Xi :=
 \{ \gamma \in 1+\CLFM{L} \cdot 2 E \bigm|
 N \gamma=1 \}.
\end{equation}
\end{theorem}
\begin{proof}
We have natural isomorphisms
\begin{equation}
 (\LatM)^\vee \cong (\CLFP{L},q), \quad
 A(\LatM) \cong
 ( \sA := \frac{\CLFP{L}}{\CLFM{L} \cdot 2 E} , q \, \bmod{2 \Z} ), \quad
\end{equation}
 where $v \in (\LatM)^\vee$ corresponds to $v \cdot 2 E \in \CLFP{L}$
 and the quadratic form $q$ on $\CLFP{L}$ is defined by
\begin{equation}
 q(x) = \frac{1}{2 \DDz} \cdot N x. 
\end{equation}
(See Section \ref{subsect:dual}.)
We observe that $\CLFM{L} \cdot 2 E$ is a two-sided ideal of
 $\CLFP{L}$ and $\sA$ becomes a ring.

For $\Phi_{\alpha,\beta} \in O(\LatM)$
 with $\alpha,\beta \in \CLFP{L}_\Q^\times$,
 we should have
\begin{equation*}
 N (\alpha \beta^{-1}) = 1, \quad
 \alpha \cdot \CLFP{L} \cdot \beta^{-1} = \CLFP{L}, \quad
 \gamma := \Phi_{\alpha,\beta}(1) = \alpha \beta^{-1} \in \CLFP{L}^\times.
\end{equation*}
Moreover, if $\Phi_{\alpha,\beta} \in O_\GAM(\LatM)$,
 then we have
 $\gamma \equiv 1$ in $\sA$.
Thus the induced action of $\Phi_{1,\gamma}$
 on $\sA \cong A(\LatM)$ is trivial,
 and hence
\begin{equation*}
 \Phi_{1,\gamma} \in O_\GAM(\LatM), \quad
 \Phi_{\alpha,\alpha} =
 \Phi_{1,\gamma} \circ \Phi_{\alpha,\beta} \in O_\GAM(\LatM).
\end{equation*}
We may assume $\alpha \in \CLFP{L}^\times$
 by replacing $\alpha$ by $c \alpha$ for some $c \in \Q^\times$ if necessary
 by Lemma \ref{lem:M_K_disc_kernel} below
 (combined with our Main Theorem).
Then $\beta = \gamma^{-1} \alpha \in \CLFP{L}^\times$.
Conversely, for $\alpha,\beta \in \CLFP{L}^\times$
 satisfying $\alpha \beta^{-1} \in \Xi$,
 we have $\Phi_{\alpha,\beta} \in O_\GAM(\LatM)$
 because $\alpha \equiv \beta$ in $\sA$.
A similar argument shows the following equivalence:
\begin{equation*}
 \alpha,\beta \in \CLFM{L}_\Q^\times, ~
 \Phi_{\alpha,\beta} \circ \tau \in O_\GAM(\LatM)
 \Leftrightarrow
 \alpha,\beta \in \CLFM{L}^\times, ~ -\alpha \beta^{-1} \in \Xi
\end{equation*}
 because we have $\tau(1)=-1$ under the isomorphism
 $(\LatM)^\vee \cong (\CLFP{L},q)$.
\end{proof}

\begin{lemma} \label{lem:M_K_disc_kernel}
Let $M$ be an even lattice and 
 $K$ be a primitive non-degenerate sublattice of $M$.
Set
\begin{equation}
 O(M,K) := \{ g \in O(M) \bigm| g(K)=K, ~ g|_{K^\bot}=\id \}.
\end{equation}
Then we have an isomorphism
\begin{equation} \label{eq:map_M_K}
 O_\GAM(M,K) := O(M,K) \cap O_\GAM(M)
 \cong O_\GAM(K); \quad g \mapsto g|_K.
\end{equation}
\end{lemma}
\begin{proof}
Let $g \in O_\GAM(M,K)$.
We shall show $g|_K \in O_\GAM(K)$.
By the exactness of the hom-functor of free $\Z$-modules,
 for $v \in K^\vee \cong \OP{Hom}_\Z(K,\Z)$,
 there exists $v' \in (K^\bot)^\vee$
 such that $v+v' \in M^\vee$.
Since $g(v+v') \equiv v + v'$ in $A(M)$, we have
\begin{equation*}
 g(v+v') = g(v) + v', \quad
 g(v)-v = g(v+v')-(v+v') \in M \cap K^\vee=K.
\end{equation*}
Thus $g|_K \in O_\GAM(K)$.
Conversely, we can extend each $g_0 \in O_\GAM(K)$
 to $g \in O(M)$ so that $g|_{K^\bot} = \id$.
Then $g$ should be contained in $O_\GAM(M,K)$.
Therefore the map in (\ref{eq:map_M_K}) is bijective.
\end{proof}

\vspace{0.5cm}

\par\noindent
Kenji Hashimoto
\par\noindent{\ttfamily kenji.hashimoto.math@gmail.com}
\break

\par\noindent
Kwangwoo Lee
\par\noindent{\ttfamily leekw@kias.re.kr}

\end{document}